\documentclass[10pt, reqno]{amsart}
\usepackage{graphicx, amssymb, amsmath, amsthm}
\numberwithin{equation}{section}
\usepackage{xcolor}
\usepackage{amscd}
\def\pa{\partial}

\let\Re=\undefined\DeclareMathOperator*{\Re}{Re}
\let\Im=\undefined\DeclareMathOperator*{\Im}{Im}

\newcommand{\R}{\mathbb{R}}
\newcommand{\C}{\mathbb{C}}
\newcommand{\wit}{\widetilde}

\newcommand{\ov}{\overline}

\newcommand{\var}{\varepsilon}

\newcommand{\norm}[1]{\| #1\|}

\newcommand{\xnorm}[4]{\|#1\|_{L_t^{#2}L_x^{#3}(#4\times\R^d)}}

\newtheorem{theorem}{Theorem}[section]

\newtheorem{lemma}[theorem]{Lemma}

\newtheorem{corollary}[theorem]{Corollary}

\newtheorem{proposition}[theorem]{Proposition}

\theoremstyle{definition}
\newtheorem{definition}[theorem]{Definition}
\newtheorem{remark}[theorem]{Remark}

\makeatletter
\newcommand{\Extend}[5]{\ext@arrow0099{\arrowfill@#1#2#3}{#4}{#5}}
\makeatother

\begin{document}
\title[Mass-critical NLS ]{Low regularity blowup solutions for the mass-critical NLS in higher dimensions}

\author{Chenmin Sun}
\address{Universit\'e C\^{o}te d'Azur, CNRS, LJAD, France}
\email{csun@unice.fr}

\author{Jiqiang Zheng}
\address{Institute of Applied Physics and Computational Mathematics, Beijing 100088}
\email{zhengjiqiang@gmail.com}

\begin{abstract}
In this paper, we study the $H^s$-stability of the log-log blowup regime (which has been completely described in a series of recent works by Merle and Rapha\"el) for the focusing mass-critical nonlinear Schr\"odinger equations $i\pa_tu+\Delta u+|u|^\frac4du=0$ in $\R^d$ with $d\geq3$. We aim to extend the result in [Colliander and  Raphael, Rough blowup solutions to the $L^2$ critical NLS, Math. Anna., 345(2009), 307-366.] for dimension two to the higher dimensions cases $d\geq3$, where we use the bootstrap argument in the above paper and the commutator estimates in [M. Visan and X. Zhang,  On the blowup for the $L^2$-critical focusing nonlinear Schr\"odinger equation in higher dimensions below the energy class. SIAM J. Math. Anal., 39(2007), 34-56.
].

\end{abstract}

 \maketitle

\begin{center}
 \begin{minipage}{100mm}
 \end{minipage}
 \end{center}

% \tableofcontents %%代表运行目录，如果不要目录，可以把这句话注释掉。

\section{Introduction}

\noindent
 We study the initial-value problem for focusing
nonlinear
 Schr\"odinger equations of the form
\begin{align} \label{equ:nls}
\begin{cases}    (i\partial_t+\Delta)u=-|u|^\frac4du,\quad
(t,x)\in\R\times\R^d,
\\
u(0,x)=u_0(x),
\end{cases}
\end{align}
where $u:\R_t\times\R_x^d\to \C$.

Equation \eqref{equ:nls} admits a number of symmetries in  $H^1(\R^d)$, explicitly:

$\bullet$ {\bf Space-time translation invariance:} if $u(t,x)$ solves \eqref{equ:nls}, then so does $u(t+t_0,x+x_0),~(t_0,x_0)\in\R\times\R^d$;

$\bullet$ {\bf Phase invariance:} if $u(t,x)$ solves \eqref{equ:nls}, then so does $e^{i\gamma}u(t,x),~\gamma\in\R;$

$\bullet$ {\bf Galilean invariance:} if $u(t,x)$ solves \eqref{equ:nls}, then for $\beta\in\R^d$, so does $e^{i\frac{\beta}2\cdot(x-\frac\beta2t)}u(t,x-\beta t)$;

$\bullet$ {\bf Scaling invariance:} if $u(t,x)$ solves \eqref{equ:nls}, then so does $u_\lambda(t,x)$ defined by
\begin{equation}\label{equ:scale}
u_\lambda(t,x)=\lambda^{\frac{d}2}u(\lambda^2t, \lambda x),\quad\lambda>0.
\end{equation}
This scaling defines a notion of \emph{criticality} for \eqref{equ:nls}. In particular, one can check that
the only homogeneous $L_x^2$-based Sobolev space that is left invariant under \eqref{equ:scale} is $L^2_x(\R^d)$, and we call problem \eqref{equ:nls} as mass-critical problem.

In above, we know that  equation \eqref{equ:nls} admits a number of symmetries in the energy space $H^1$: if $u(t,x)$ solves \eqref{equ:nls}, then for any $(\lambda_0,t_0,x_0,\beta_0,\gamma_0)\in\R^+\times\R\times\R^d\times\R^d\times\R$, so does
\begin{equation}\label{equ:symen}
v(t,x)=\lambda_0^\frac{d}2e^{i\gamma_0}e^{i\frac{\beta_0}2(x-\frac{\beta_0}2t)}u\big(\lambda_0^2t+t_0,\lambda_0x+x_0-\beta_0t\big).
\end{equation}
From the Ehrenfest law or direct computation, these symmetries induce invariances in the energy space, namely
\begin{align}\label{equ:mass123}
&\textbf{Mass:}~~M(u)=\int_{\R^d}|u(t,x)|^2\;dx=M(u_0);\\\label{equ:energy123}
&\textbf{Energy:}~~E(u)=\int_{\R^d}\Big(\frac12|\nabla u(t,x)|^2-\frac{d}{2(d+2)}|u(t,x)|^{\frac{2(d+2)}{d}}\Big)\;dx=E(u_0);\\\label{equ:momentum123}
&\textbf{Momentum:}~~P(u)={\rm Im}\int_{\R^d}\nabla u\bar{u}\;dx=P(u_0).
\end{align}

In the focusing nonlinear Schr\"odinger equation, the special solutions play an important role. They are the so-called
solitary waves and are of the form $u(t,x)=e^{iwt}Q_w(x)$ (which is a global solution but not scatters), where $Q_w$ solves
\begin{equation}\label{equ:ellmasub12}
\begin{cases}
\Delta Q_w+Q_w|Q_w|^{\frac4d}=wQ_w,\\
Q_w\in H^1(\R^d)\backslash\{0\}.
\end{cases}
\end{equation}
Equation \eqref{equ:ellmasub12} is a standard nonlinear elliptic equation.
It is known that if $w\leq0$, then \eqref{equ:ellmasub12} does not have any solution. Therefore, we assume that $w>0$.

In dimension $d=1$,
there exists a unique solution in $H^1$ up to translation to \eqref{equ:ellmasub12} and infinitely
many with growing $L^2$-norm for $d\geq2$. Nevertheless, from \cite{BL,GNN,Kw}, there is a unique positive solution up to translation $Q_w(x)$.
$Q_w$ is in addition radially symmetric.  Letting $Q=Q_{w=1}$, then $Q_w(x)=w^{\frac1{p-1}}Q(w^{1/2}x)$ from scaling property, i.e. $Q$ solves
\begin{equation}\label{equ:ellmasub123}
\Delta Q+Q|Q|^{p-1}=Q.
\end{equation}

Using the Strichartz estimate and a standard fixed point argument, see Ginibre and Velo\cite{GV79} and Cazenave and Weissler\cite{CaW}, we derive that \eqref{equ:nls} is locally well-posedness in $H^s$ for $0\leq s\leq1$ and the Cauchy problem is subcritical in $H^s$ for $s>0$: for $u_0\in H^s,~s>0$, there exists $0<T\leq +\infty$ such that $u\in C([0,T),H^s)$ and either $T=+\infty$ and we say the solution is global, or $T<+\infty$ and then
$$\limsup_{t\to T}\|u(t)\|_{\dot{H}^s}=+\infty$$
and we say the solution blows up in finite time.

In the case $\|u_0\|_{L^2}<\|Q\|_{L^2}$, from classical variational arguments, one can obtain the global well-posedness in $H^1(\R^d)$.
Indeed, this follows from the conservation of the energy, the mass and Gagliardo-Nirenberg inequality as exhibited by
Weinstein in \cite{wen82}:
\begin{equation}
\forall~u\in H^1:~E(u)\geq\frac12\Big(\int|\nabla u|^2\Big)\Big(1-\Big(\frac{\int|u|^2}{\int Q^2}\Big)\Big).
\end{equation}
Moreover, the scattering result in $L^2(\R^d)$ is obtained by Killip, Tao, Visan and Zhang \cite{KTV2009,KVZ2008} for radial initial data and Dodson \cite{Dod}
for nonradial initial data.

In the case $\|u_0\|_{L^2}=\|Q\|_{L^2},$  the
pseudo-conformal transformation applied to the stationary solution $e^{it}Q$  yields an explicit solution
\begin{equation}\label{equ:stx}
S(t,x)=\frac1{|t|^{d/2}}Q\Big(\frac{x}{t}\Big)e^{-i\frac{|x|^2}{4t}+\frac{i}{t}},~\|S(t)\|_{L^2}=\|Q\|_{L^2}
\end{equation}
which scatters as $t\to-\infty$ in the sense that there exists $\psi_-\in L^2(\R^d)$ such that
$$\lim_{t\to-\infty}\|S(t,x)-e^{it\Delta}\psi_-\|_{L^2(\R^d)}=0.$$
And
$S(t,x)$ blows up at $T=0$ at the speed
 $$\|\nabla S(t)\|_{L^2}\sim\frac1{|t|}.$$ An essential feature of \eqref{equ:stx} is compact up to the symmetries of the flow, meaning
that all the mass goes into the singularity formation
  \begin{equation}\label{equ:stdels0}
  |S(t)|^2\rightharpoonup\|Q\|_{L^2}^2\delta_{x=0}~\text{as}~t\to0.
  \end{equation}
  It turns out that $S(t)$ is the unique minimal mass blow-up solution in $H^1$ in the following sense: let $u(-1)\in H^1$ with $\|u(-1)\|_{L^2}=\|Q\|_{L^2}$, and assume that $u(t)$ blows up at $T=0$, then $u(t)=S(t)$ up to the symmetries of the equation, see  Merle \cite{Merle92} (radial) and \cite{Merle93} (general case). Note
that from direct computation
$$E(S(t,x))>0,~\text{and}~\|\nabla S(t)\|_{L^2}=\frac{C}{|t|}.$$
The general intuition is that such a behavior is exceptional in the sense that such minimal
elements can be classified.

The situation $\|u_0\|_{L^2}>\|Q\|_{L^2}$ has been clarified by Merle and Rapha\"el in the series of papers \cite{MR05annmath,MR03GAFA,MR04Invemath,MR06JAMS,MR05CMP,Ra05Mathann}. Let us define the differential operator
$$\Lambda:=\frac{d}2+y\cdot\nabla,$$
which will be of constant use. Then we introduce the following property:

{\bf Spectral property.} Let $d\geq1$.  Consider the two real Schr\"odinger operators
\begin{equation}\label{equ:scroper}
\mathcal{L}_1=-\Delta+\frac2d\Big(\frac4d+1\Big)Q^{\frac4d-1}y\cdot\nabla Q,~\mathcal{L}_2=-\Delta+\frac2d
Q^{\frac4d-1}y\cdot\nabla Q,
\end{equation}
and the real quadratic form for $\varepsilon=\varepsilon_1+i\varepsilon_2\in H^1:$
$$H(\varepsilon,\varepsilon)=(\mathcal{L}_1\varepsilon_1,\varepsilon_1)+(\mathcal{L}_2\varepsilon_2,
\varepsilon_2).$$
Then there exists a universal constant $\tilde{\delta}_1>0$ such that for all $\varepsilon\in H^1$, if
$$(\varepsilon_1,Q)=(\varepsilon_1,Q_1)=(\varepsilon_1,yQ)=(\varepsilon_2,Q_1)=(\varepsilon_2,Q_2)
=(\varepsilon_2,\nabla Q)=0,$$
then
$$H(\varepsilon,\varepsilon)\geq\tilde{\delta}_1\int\Big(|\nabla \varepsilon|^2+|\varepsilon|^2e^{-(2_-)|y|}\Big)$$
where  $Q_1=\Lambda  Q$ and $Q_2=\Lambda Q_1$, and $2_-=2-\epsilon$ with $0<\epsilon\ll1.$

\begin{remark}
We remark that this spectral property has been proved rigorously in \cite{MR05annmath} for dimension $d=1$, since the ground state
$Q$
is explicit in dimension one and the spectral property
could be deduced from some known properties of the second-order differential operators  \cite{Titch46}.  For the dimensions
$d\in\{2,3,4\}$,  Fibich,  Merle and Raphael\cite{FMR} gave a numerically-assisted proof of the above
 spectral  property by using  the  numerical  representation of the ground state $Q$. Rectently, Yang,
   Roudenko and  Zhao \cite{YRZ} showed the spectral property in dimensions $5\leq d\leq10$ (general case) and $d\in\{11,12\}$ (radial).
\end{remark}

Based on the works \cite{MR05annmath,MR03GAFA,MR04Invemath,MR06JAMS,MR05CMP,Ra05Mathann}, we now have:

\begin{theorem}[Dynamics of NLS]\label{thm:dynamssnls}
Let $d\geq1$
 and assume the spectral property holds true. Then there exists $\alpha^\ast>0$ and a universal constant $C^\ast>0$ such that the following is true. Let $u_0\in \mathcal{B}_{\alpha^\ast}$ with
 $$\mathcal{B}_{\alpha^\ast}:=\big\{u_0\in H^1(\R^d):~\int Q^2\leq\int |u_0|^2<\int Q^2+\alpha^\ast\big\},$$ let $u(t)$ be the corresponding solution to \eqref{equ:nls} with $[0,T)$ its maximum time interval with existence on the right in $H^1$.

$(i)$ {\bf Estimates on the blow-up speed:} assume $u(t)$ blows up in finite time i.e.,
$0<T<+\infty$, for $t$ close enough to $T$ , we have either
\begin{equation}\label{equ:loglogbsp}
\lim_{t\to T}\frac{\|\nabla u\|_{L_x^2}}{\|Q\|_{L_x^2}}\Big(\frac{T-t}{\ln|\ln(T-t)|}\Big)^\frac12=\frac1{\sqrt{2\pi}}
\end{equation}
or
\begin{equation}\label{equ:stblow}
\|\nabla u(t)\|_{L^2}\geq\frac{C^\ast}{(T-t)\sqrt{E^G(u_0)}},
\end{equation}
with
$$E^G(u):=E(u)-\frac12\frac{P(u)^2}{\|u\|_{L_x^2}^2}.$$

$(ii)$ {\bf Description of the singularity:} assume $u(t)$ blows up in finite time, then there exist parameters $(\lambda(t),x(t),\gamma(t))\in \R^+\times\R^d\times\R$ and an asymptotic profile
$u^\ast\in L^2(\R^d)$ such that
\begin{equation}\label{equ:l2asyt123}
u(t)-\frac1{\lambda(t)^{d/2}}Q\Big(\frac{x-x(t)}{\lambda(t)}\Big)e^{i\gamma(t)}\to u^\ast\quad\text{in}\quad L^2(\R^d)\quad\text{as}\quad t\to T.
\end{equation}
Moreover, the blow up point is finite in the sense that
$$x(t)\to x(T)\in\R^d,\quad \text{as}\quad t\to T.$$
Moreover, assume $u(t)$ satisfies \eqref{equ:loglogbsp}, $x(T)$ be its blow up point. Set
\begin{equation}
\lambda_0(t)=\sqrt{2\pi}\sqrt{\frac{T-t}{\ln|\ln(T-t)|}}
\end{equation}
then there exists a phase parameter $\gamma_0(t)\in\R$ such that:
\begin{equation}\label{equ:l2asyt1231}
u(t)-\frac1{\lambda_0(t)^{d/2}}Q\Big(\frac{x-x_0(T)}{\lambda(t)}\Big)e^{i\gamma_0(t)}\to u^\ast\quad\text{in}\quad L^2(\R^d)\quad\text{as}\quad t\to T.
\end{equation}

$(iii)$ {\bf Universality of blow up profile in $\dot{H}^1$}: if we assume that $u(t)$ blows up in finite time with \eqref{equ:loglogbsp}, then there exist parameters  $\lambda_0(t)=\frac{\|\nabla Q\|_{L_x^2}}{\|\nabla u(t)\|_{L_x^2}},~x_0(t)\in \R^d$ and $\gamma_0(t)\in\R$ such that
\begin{equation}\label{equ:conuin123}
e^{i\gamma_0(t)}\lambda_0(t)^\frac{d}2u(t,\lambda_0(t)x+x_0(t))\to Q\quad\text{in}\quad \dot{H}^1,\quad \text{as}\quad t\to T.
\end{equation}
If  $u(t)$ satisfies \eqref{equ:stblow}, then  asymptotic stability \eqref{equ:conuin123} holds on a sequence $t_n\to T.$

$(iv)$ {\bf Sufficient condition for log-log blow-up:} if \begin{equation}\label{equ:Egdef}
E^G(u):=E(u)-\frac12\frac{P(u)^2}{\|u\|_{L_x^2}^2}<0,
\end{equation} then $u(t)$ blows up
in finite time with the log-log speed \eqref{equ:loglogbsp}. More generally, the set of initial data
$u_0\in \mathcal{B}_{\alpha^\ast}$ such that the corresponding solution $u(t)$ to \eqref{equ:nls} blows up in finite time
$0<T<+\infty$ with the log-log speed \eqref{equ:loglogbsp} is open in $H^1$ ({\bf $H^1$ stability of the log-log regime}).

$(v)$ {\bf Asymptotic of $u^\ast$ on the singularity:} assume $T<+\infty$; if $u(t)$ satisfies \eqref{equ:loglogbsp}, then  for $R>0$ small,
\begin{equation}\label{equ:uastl23}
\frac1{C^\ast(\ln|\ln(R)|)^2}\leq\int_{|x-x(T)|<R}|u^\ast(x)|^2\;dx\leq\frac{C^\ast}{(\ln|\ln(R)|)^2}
\end{equation}
which implies $
u^\ast\not\in H^1~\text{and}~u^\ast\not\in L^p~\text{with}~p>2.
$ If $u(t)$ satisfies \eqref{equ:stblow}, then \begin{equation}\label{equ:lstlwh1}
\int_{|x-x(T)|\leq R}|u^\ast(x)|^2\;dx\leq C^\ast E_0R^2,~\text{and}~u^\ast\in H^1.
\end{equation}

\end{theorem}

\begin{remark}
The above theorem asserts the existence and the stability of a log-log blowup regime, and gives sufficient conditions to ensure its occurrence. It also assets that the log-log blowup regime is open in $H^1$.
 In \cite{CR09}, Colliander and Raphael proved that the log-log blowwup dynamics described by Theorem \ref{thm:dynamssnls} are stable under small $H^s$ perturbations with $0<s\leq1$ in dimension $d=2$. In this paper, we extend their result to higher dimensional cases $d\geq3.$

\end{remark}

Now, we state our main result.

\begin{theorem}[$H^s$-stability of the log-log regime]\label{thm:main}
Let $d\geq3$ and $\frac1{1+\min\{1,\frac4d\}}<s\leq1$,  and assume the spectral property holds true. Then, the log-log blowup dynamics described by Theorem \ref{thm:dynamssnls} are stable under small $H^s$ perturbations. In other words, let $u_0\in H^1$ evolve into  a log-log blowup solution given by Theorem \ref{thm:dynamssnls}. Then, there exists $\varepsilon=\varepsilon(s,u_0)$ such that for any $v_0\in H^s$ with
$$\|v_0-u_0\|_{H^s(\R^d)}<\varepsilon,$$
then, the corresponding solution $v(t)$ to \eqref{equ:nls} blows up in finite time $0<T<+\infty$ with the following blowup dynamics: there exist geometrical parameters $(\lambda(t),x(t),\gamma(t))\in(0,+\infty)\times\R^d\times\R$ and an asymptotic residual profile $v^\ast\in L^2$, with $v^\ast\not\in L^p(p>2)$, such that as $t\to T$
\begin{align}\label{equ:convt}
v(t)-\frac1{\lambda(t)^\frac{d}2}Q\Big(\frac{x-x(t)}{\lambda(t)}\Big)e^{i\gamma(t)}\to v^\ast\quad \text{in}\quad L^2,\\\label{equ:xtconx}
x(t)\to x(T)\in\R^d,\\\label{equ:lamdts}
\lambda(t)\sqrt{\frac{\log|\log(T-t)|}{T-t}}\to\sqrt{2\pi}.
\end{align}

\end{theorem}

\begin{remark}
The restriction $\frac1{1+\min\{1,\frac4d\}}<s$ follows from the commutator estimation in Lemma \ref{commutor}.

\end{remark}

\subsection{Outline of proof of Theorem \ref{thm:main}}

First, we recall from \cite{MR03GAFA,MR04Invemath} the existence of a one parameter family of localized self-similar profiles close to the ground state solution $Q$.

\begin{proposition}[Localized self-similar profiles, \cite{MR03GAFA,MR04Invemath}]
There exist universal constants $C>0,~\eta^\ast>0$ such that the following holds true. For all $0<\eta<\eta^\ast$, there exist constants $\varepsilon^\ast(\eta)>0,~b^\ast(\eta)>0$ going to zero as $\eta\to0$ such that for all $|b|<b^\ast(\eta)$, there exists a unique radial solution
$\tilde{Q}_b$ to
\begin{equation}
\begin{cases}
\Delta\tilde{Q}_b-\tilde{Q}_b+ib\Lambda\tilde{Q}_b+\tilde{Q}_b|\tilde{Q}_b|^\frac4d=0,\\
P_b=\tilde{Q}_b e^{i\frac{b|y|^2}4}>0~\text{in}~B_{R_b},\\
|\tilde{Q}_b-Q(0)|<\varepsilon^\ast(\eta),~\tilde{Q}_b(R_b)=0,
\end{cases}
\end{equation}
with $R_b=\frac2{|b|}\sqrt{1-\eta},~B_{R_b}=\{y\in\R^d,~|y|\leq R_b\}.$ Let $\phi_b(y)$ be a regular radially symmetric cut-off function satisfying
\begin{equation*}
  \phi_b(y)=\begin{cases}
  1 \quad \text{if}\quad |x|\leq R_b^-=\sqrt{1-\eta}R_b,\\
  0\quad \text{if}\quad |x|\geq R_b,
  \end{cases} ~0\leq\phi_b\leq1,
\end{equation*}
and $\|\nabla\phi_b\|_{L^\infty}+\|\Delta\phi_b\|_{L^\infty}\to0$ as $|b|\to0.$ Moreover, let
\begin{equation}\label{def:qb}
Q_b(r)=\tilde{Q}_b(r)\phi_b(r),
\end{equation}
then
\begin{align}\label{equ:qbclsq}
\big\|e^{Cr}(Q_b-Q)\big\|_{H^{10}\cap C^2}\to0\quad\text{as}\quad |b|\to0,\\\label{equ:paqbest}
\big\|e^{cr}\big(\tfrac{\pa Q_b}{\pa b}+i\tfrac{|y|^2}4Q\big)\big\|_{C^2}\to0\quad \text{as}\quad |b|\to0,\\\label{equ:eqb}
|E(Q_b)|\leq e^{-\frac{C}{|b|}},
\end{align}
and $Q_b$ has super-critical mass:
\begin{equation}\label{equ:qbcrimass}
0<\frac{d}{d(b^2)}\Big(\int|Q_b|^2\Big)\Big|_{b^2=0}=d_0<+\infty.
\end{equation}

\end{proposition}

\begin{remark}\label{Rem:qbequ}
The profiles $Q_b$ are not exact self-similar solutions and we define the error term $\Psi_b$ by:
\begin{equation}\label{def:psib}
\Delta Q_b-Q_b+ib\Lambda Q_b+Q_b|Q_b|^\frac4d=-\Psi_b.
\end{equation}
Indeed,
$$-\Psi_b=2\nabla\tilde{Q}_b\cdot\nabla \phi_b+\tilde{Q}_b(\Delta\phi_b)+ib\tilde{Q}_b(\Lambda\phi_b)+(\phi_b^{1+\frac4d}-\phi_b)|\tilde{Q}_b|^\frac4d\tilde{Q}_b.$$
\end{remark}

Next, we introduce the outgoing radiation escaping the soliton core according to the following Lemma.

\begin{lemma}[Linear outgoing radiation, \cite{MR04Invemath}(Lemma 15) ]\label{lem:linoutrad}
There exist universal constants $C>0$ and $\eta^\ast>0$ such that $\forall~0<\eta<\eta^\ast$, there exists $b^\ast>0$ such that for any $|b|<b^\ast$, the following holds true: with $\Psi_b$ given by \eqref{def:psib}, there exists a unique radiation solution $\zeta_b$ to
\begin{equation}
\begin{cases}
\Delta \zeta_b-\zeta_b+ib\Lambda\zeta_b=\Psi_b,\\
\int|\nabla \zeta_b|^2<+\infty.
\end{cases}
\end{equation}
Moreover, let
\begin{equation}\label{equ:gammab}
\Gamma_b:=\lim_{|y|\to+\infty}|y|^d|\zeta_b(y)|^2,
\end{equation}
then,
\begin{equation}\label{equ:geb}
e^{-(1+C\eta)\frac{\pi}{|b|}}\leq\Gamma_b\leq e^{-(1-C\eta)\frac{\pi}{|b|}}.
\end{equation}
More precisely, it follows that
$$ \left\||y|^{\frac{d}{2}}(|\zeta_b|+|y||\nabla\zeta_b|)\right\|_{L^{\infty}(|y|\geq R_b)}\leq \Gamma_b^{\frac{1}{2}-C\eta},\quad \int |\nabla\zeta_b|^2\leq \Gamma_b^{1-C\eta}.
$$
For $|y|$ small, we have: For any $\sigma\in (0,5)$, there exists $\eta^{**}(\sigma)$, such that for any $0<\eta<\eta^{**}$, there exists $b^{**}(\eta)$ such that for any $0<|b|<b^{**}(\eta)$, it follows that
$$ \|\zeta_be^{-\frac{\sigma\theta(b|y|)}{|b|}}\|_{\mathcal{C}^2(|y|\leq R_b)}\leq \Gamma_b^{\frac{1}{2}+\frac{\sigma}{10}}.
$$
Last, $\zeta_b$ is differentiable with respect to $b$ with estimate
$$ \|\partial_b\zeta_b\|_{\mathcal{C}^1}\leq \Gamma_b^{\frac{1}{2}-C\eta}.
$$
\end{lemma}

Next, let us introduce some notations in  the I-method, which consists in  smoothing out the $H^s$-initial data with $0<s<1$ in order
to access a good local theory available at the $H^1$-regularity. To do it, for $N\gg1$, we define the Fourier multiplier $I_N$ by
$$\widehat{I_Nu}(\xi):=m_N(\xi)\hat{u}(\xi),$$ where $m_N(\xi)$ is a
smooth radial decreasing cut off function such that
\begin{equation}\label{mxidy}
m_N(\xi)=\begin{cases} 1, ~~  \qquad\qquad|\xi|\leq N,\\
\Big(\frac{|\xi|}{N}\Big)^{s-1},\quad |\xi|\geq2N.
\end{cases}
\end{equation}
Thus, $I_N$ is the identity operator on frequencies $|\xi|\leq N$ and
behaves like a fractional integral operator of order $1-s$ on higher
frequencies. It is easy to show  that the operator $I_N$ maps $H^s$ to
$H^1$. Moreover, we have
\begin{equation}\label{equcont}
 \|u\|_{H^s}\lesssim\|I_Nu\|_{H^1}\lesssim
N^{1-s}\|u\|_{H^s}.
\end{equation}
Let $\delta_\lambda f(x)=f(x/\lambda)$. Then, by a simple computation, we have
\begin{equation}\label{equ:scaio}
(I_N\delta_\lambda f)(x)=(\delta_\lambda I_{N\lambda}f)(x).
\end{equation}

By a simple argument as in \cite{CR09} and perturbation theory, we can reduce Theorem \ref{thm:main} to  the following proposition.

\begin{proposition}[Explicit description of the blowup set]\label{prop:main}
Let $s>s(d)$ and consider an initial data
\begin{equation}\label{equ:iniassum}
u_0=G(0)+H(0),~G(0)\in H^1,~H(0)\in H^s
\end{equation}
such that $G(0)$ admits a geometrical decomposition:
\begin{equation}\label{equ:g0x}
G(0,x)=\frac1{\lambda(0)^\frac{d}2}\big(Q_{b(0)}+g(0)\big)\Big(\frac{x-x(0)}{\lambda(0)}\Big)e^{-i\gamma(0)},
\end{equation}
with the following controls:

$(i)$ {\bf Control of the scaling parameter:}
\begin{equation}\label{equ:b0l0}
0<b(0)\ll1,~~0<\lambda(0)\leq e^{-e^\frac{2\pi}{3b(0)}},
\end{equation}

$(ii)$ {\bf $L^2$-control of the excess of mass:}
\begin{equation}\label{equ:excmass}
\|g(0)\|_{L^2}+\|H(0)\|_{L^2}\ll1,
\end{equation}

$(iii)$ {\bf $H^s$-control of the rough excess of mass:}
\begin{equation}\label{equ:h0hs}
\|H(0)\|_{H^s}\leq \lambda(0)^{10},
\end{equation}

$(iv)$ {\bf $H^1$-smallness of $g(0)$:}
\begin{equation}\label{equ:g0small}
\int|\nabla g(0)|^2+\int |g(0)|^2e^{-|y|}\leq \Gamma_{b(0)}^\frac34,
\end{equation}

$(v)$ {\bf Control of the conservation laws for the $H^1$-part:}
\begin{align}\label{equ:eg0}
|E(G(0))|\leq&\frac1{\sqrt{\lambda(0)}},\\\label{equ:gg0}
|P(G(0))|\leq&\frac1{\sqrt{\lambda(0)}}.
\end{align}
Then, the corresponding $H^s$ solution $u(t)$ to \eqref{equ:nls} blows up in finite time $0<T<+\infty$ in the log-log regime and the conclusions of Theorem \ref{thm:main} hold true.

\end{proposition}

Hence, our goal is to prove Proposition \ref{prop:main}.
Let $u_0\in H^s$ satisfy the hypotheses of Proposition \ref{prop:main}. We can rewrite the decomposition \eqref{equ:iniassum} as
\begin{equation}\label{equ:reu0}
u(0,x)=\frac1{\lambda(0)^\frac{d}2}\big(Q_{b(0)}+\varepsilon(0)\big)\Big(\frac{x-x(0)}{\lambda(0)}\Big)e^{-i\gamma(0)}
\end{equation}
with $\varepsilon(0)=g(0)+h(0)$ and
\begin{equation}\label{equ:defh0}
H(0,x)=\frac1{\lambda(0)^\frac{d}2}h\Big(0,\frac{x-x(0)}{\lambda(0)}\Big)e^{-i\gamma(0)}.
\end{equation}
We have by \eqref{equ:h0hs}
\begin{equation}\label{equ:h0small}
\|h(0)\|_{\dot{H}^s}=\lambda(0)^s\|H(0)\|_{\dot{H}^s}\leq \lambda(0)^{10+s}.
\end{equation}
This together with \eqref{equ:excmass} and \eqref{equ:g0small} yields that
\begin{equation}\label{equ:vesma}
\|\varepsilon(0)\|_{H^s}\leq\|g(0)\|_{H^s}+\|h(0)\|_{H^s}\ll1.
\end{equation}
Next, we derive a frequency localized version of \eqref{equ:g0small} for $\varepsilon(0)$. Set
\begin{equation}\label{equ:n0def}
N(0)=\Big(\frac1{\lambda(0)}\Big)^\frac1{\beta},
\end{equation}
with $\beta$ given in Remark \ref{rem:beta}, then,
\begin{equation}\label{equ:nlamd}
1\ll\Big(\frac1{\lambda(0)}\Big)^\frac{1-\beta}{\beta}=N(0)\lambda(0).
\end{equation}
And so, using \eqref{equ:geb}, \eqref{equcont}, \eqref{equ:b0l0} and \eqref{equ:h0small}, we get
\begin{equation*}
\int|I_{N(0)\lambda(0)}\nabla h(0)|^2\lesssim (N(0)\lambda(0))^{2(1-s)}\|h(0)\|_{\dot{H}^s}^2\lesssim \lambda(0)^{10}\leq\Gamma_{b(0)}^{10}.
\end{equation*}
This together with \eqref{equ:h0hs} and \eqref{equ:g0small} implies
\begin{equation}\label{equ:ncsma}
\int|I_{N(0)\lambda(0)}\nabla \varepsilon(0)|^2+\int|\varepsilon(0)|^2e^{-|y|}\leq\Gamma_{b(0)}^\frac34.
\end{equation}
\begin{remark}\label{rem:beta}
The restriction on $\beta$ stems from two sides. One comes from Corollary \ref{cor:modlwp} below: $$\frac{4s}{\min\{4,d\}s-4(1-s)}<\frac1{\beta}.$$
Another comes from  Lemma \ref{lem:finallemm} below
$$\frac{1}{\beta}>\frac{4s}{\min\{4,d\}s^2-(1-s)}$$
to guarantee the convergence of the summation. Hence, we will take $\beta\in(0,1)$ such that
$$\frac{1}{\beta}>\min\Big\{\frac{4s}{\min\{4,d\}s-4(1-s)}\frac{4s}{\min\{4,d\}s^2-(1-s)}\Big\}.$$

\end{remark}

\begin{lemma}[Nonlinear modulation theory, \cite{MR05annmath,MR03GAFA,MR06JAMS}]\label{lem:moduthe}
There exists $\alpha_2>0$ such that for $\alpha_0<\alpha_2$, there exist some functions $(\lambda,\gamma,x,b):~[0,T)\to(0,+\infty)\times\R\times\R^d\times\R$ such that
\begin{equation}\label{equ:gemdec}
\varepsilon(t,y)=e^{i\gamma(t)}\lambda(t)^\frac{d}2u\big(t,\lambda(t)y+x(t)\big)-Q_{b(t)}(y)
\end{equation}
satisfies the following orthogonality conditions:
\begin{align}\label{equ:cy2}
(\varepsilon_1(t),~|y|^2\Sigma)+(\var_2(t),~|y|^2\Theta)=&0,\\\label{equ:cy1}
(\var_1(t),~y\Sigma)+(\var_2(t),~y\Theta)=&0,\\\label{equ:cth}
-(\var_1(t),~\Lambda\Theta)+(\var_2(t),~\Lambda\Sigma)=&0,\\\label{equ:cth2}
-(\var_1(t),~\Lambda^2\Theta)+(\var_2(t),~\Lambda^2\Sigma)=&0,
\end{align}
where $\var=\var_1+i\var_2$ and $Q_{b(t)}=\Sigma+i\Theta$ in terms of real and imaginary part.

\end{lemma}

\begin{remark}
The existence of such a decomposition \eqref{equ:gemdec} requires only the smallness of the local $L^2$-norm of $\var$ due to the regularity of $Q_b$ and its fast decay in space. We note that \eqref{equ:ncsma} ensures that the deformed parameters ensuring the orthogonality conditions at time $t=0$ are exponentially small in $b(0)$ compared to \eqref{equ:reu0}. We shall thus abuse notations at time $t=0$ and identify these two decompositions which satisfy the initialized control of Proposition \ref{prop:main}.
\end{remark}

Our main claim now is that the controls of Proposition \ref{prop:main} determine a trapped dynamical region. In other words, we claim the following bootstrapped estimates. Consider a time interval $[0,T^+]$ such that the solution $u(t)$ admits a decomposition
\begin{equation}\label{equ:decass}
u(t,x)=\frac1{\lambda(t)^\frac{d}2}\big(Q_{b(t)}(\cdot)+\var(t,\cdot)\big)\Big(\frac{x-x(t)}{\lambda(t)}\Big)e^{-i\gamma(t)},~t\in[0,T^+]
\end{equation}
satisfying the orthogonality conditions \eqref{equ:cy2}-\eqref{equ:cth2}. Now, let us assume the following uniform controls on $[0,T^+]$:

$(i)$ {\bf Control of $b(t)$ and the $L^2$ mass:}
\begin{equation}\label{equ:btvart}
b(t)>0\quad\text{and}\quad b(t)+\|\var(t)\|_{L^2}\leq10\big(b(0)+\|\var(0)\|_{L^2}\big);
\end{equation}

$(ii)$ {\bf Control and monotonicity of the scaling parameter:}
\begin{equation}\label{equ:ltass}
\lambda(t)\leq e^{-e^\frac{\pi}{100b(t)}}
\end{equation}
and almost monotonicity:
\begin{equation}\label{equ:almons}
\forall~0<t_1\leq t_2\leq T^+,~\lambda(t_2)\leq \frac32\lambda(t_1).
\end{equation}
Let $k_0\leq k^+$ integers such that
\begin{equation}\label{equ:l0k0}
\frac1{2^{k_0}}\leq\lambda(0)\leq\frac1{2^{k_0-1}},~\frac1{2^{k^+}}\leq\lambda(T^+)\leq\frac1{2^{k^+-1}},
\end{equation}
and for $k_0\leq k\leq k^+$, let $t_k$ be a time such that
\begin{equation}\label{equ:ltk}
\lambda(t_k)=\frac1{2^k},
\end{equation}
then, we assume the control of the doubling time interval:
\begin{equation}\label{equ:dotime}
t_{k+1}-t_k\leq k\lambda(t_k)^2.
\end{equation}

$(iii)$ {\bf Frequency localized control of the excess of mass:} let
\begin{equation}\label{equ:ntass}
N(t)=\Big(\frac1{\lambda(t)}\Big)^\frac1{\beta},
\end{equation}
then,
\begin{equation}\label{equ:vartass}
\int|I_{N(t)\lambda(t)}\nabla \var(t)|^2+\int|\var(t)|^2e^{-|y|}\leq \Gamma_{b(t)}^\frac14.
\end{equation}

We then claim the following Lemma which is the main step of the proof of Proposition \ref{prop:main} and states that all above estimates may be improved:
\begin{lemma}[Bootstrap lemma]\label{lem:bootstrap}
There holds the following uniform control on $[0,T^+]$:
\begin{align}\label{equ:btimp}
b(t)>0\quad\text{and}\quad b(t)+\|\var(t)\|_{L^2}\leq&5\big(b(0)+\|\var(0)\|_{L^2}\big),\\\label{equ:lamtimp}
\lambda(t)\leq& e^{-e^\frac{\pi}{10b(t)}},\\\label{equ:lmt2imp}
\forall~0<t_1\leq t_2\leq T^+,~\lambda(t_2)\leq& \frac54\lambda(t_1),\\\label{equ:tktk-1}
t_{k+1}-t_k\leq& \sqrt{k}\lambda(t_k)^2,\\\label{equ:lamtimpse}
\int|I_{N(t)\lambda(t)}\nabla \var(t)|^2+\int|\var(t)|^2e^{-|y|}\leq& \Gamma_{b(t)}^\frac23.
\end{align}
\end{lemma}

\begin{remark}\label{rem:uths}
By \eqref{equcont}, and the bootstrapped estimates \eqref{equ:btvart}, \eqref{equ:ntass}, \eqref{equ:vartass}, we obtain
\begin{equation}\label{equ:varsmall}
\|\var(t)\|_{H^s}\lesssim\|I_{N(t)\lambda(t)}\var(t)\|_{H^1}\ll1.
\end{equation}
This together with the geometrical decomposition \eqref{equ:decass}, \eqref{equ:qbclsq}  yields that
\begin{equation}\label{equ:uthslam}
\|u(t)\|_{H^s}=\frac{\|Q_{b(t)}+\var(t)\|_{H^s}}{\lambda(t)^s}\sim\frac1{\lambda(t)^s}.
\end{equation}

\end{remark}

%%%%%%%%%%%%%%%%%%%%%%%%%%%%%%%%%%%%%%%%%%%%%%%%%%%%%%%%%%%%%%%%%%%%%%%%%%%%%%%%%%%%%%%%%%%%%%%%%%%%%%%%%%%%%%%%%%%%%%%%%%%%%%%%%%%%%%%%%%%%%%

%%%%%%%%%%%%%%%%%%%%%%%%%%%%%%%%%%%                                %%%%%%%%%%%%%%%%%%%%%%%%%%%%%%%%%%%%%%%%%%%%%%%%%%%%%%%%%%%%%%%%%%%%%%

%%%%%%%%%%%%%%%%%%%%%%%%%%%%%%%%%%%%%%%%%%%%%%%%%%%%%%%%%%%%%%%%%%%%%%%%%%%%%%%%%%%%%%%%%%%%%%%%%%%%%%%%%%%%%%%%%%%%%%%%%%%%%%%%%%%%%%%%%%%%%%

\section{Notation and Almost conservation law}
\subsection{Some notation}
For nonnegative quantities $X$ and $Y$, we will write $X\lesssim Y$ to denote the estimate $X\leq C Y$ for some $C>0$. If $X\lesssim Y\lesssim X$, we will write $X\sim Y$. Dependence of implicit constants on the power $p$ or the dimension will be suppressed; dependence on additional parameters will be indicated by subscripts. For example, $X\lesssim_u Y$ indicates $X\leq CY$ for some $C=C(u)$. We denote $a_{\pm}$ as $a\pm\epsilon$ with $0<\epsilon\ll1.$

For a spacetime slab $I\times\R^d$, we write $L_t^q L_x^r(I\times\R^d)$ for the Banach space of functions $u:I\times\R^d\to\C$ equipped with the norm
    $$\xnorm{u}{q}{r}{I}:=\bigg(\int_I \norm{u(t)}_{L_x^r(\R^d)}\bigg)^{1/q},$$
with the usual adjustments when $q$ or $r$ is infinity. When $q=r$, we abbreviate $L_t^qL_x^q=L_{t,x}^q$.
We will also often abbreviate $\norm{f}_{L_x^r(\R^d)}$ to $\norm{f}_{L_x^r}.$ For $1\leq r\leq\infty$,
we use $r'$ to denote the dual exponent to $r$, i.e. the solution to $\tfrac{1}{r}+\tfrac{1}{r'}=1.$

The Fourier transform on $\mathbb{R}^d$ is defined by
\begin{equation*}
\aligned \widehat{f}(\xi):= \big( 2\pi
\big)^{-\frac{d}{2}}\int_{\mathbb{R}^d}e^{- ix\cdot \xi}f(x)dx ,
\endaligned
\end{equation*}
giving rise to the fractional differentiation operators
$|\nabla|^{s}$ and $\langle\nabla\rangle^s$,  defined by
\begin{equation*}
\aligned
\widehat{|\nabla|^sf}(\xi):=|\xi|^s\hat{f}(\xi),~~\widehat{\langle\nabla\rangle^sf}(\xi):=\langle\xi\rangle^s\hat{f}(\xi),
\endaligned
\end{equation*} where $\langle\xi\rangle:=1+|\xi|$.
This helps us to define the homogeneous and inhomogeneous Sobolev
norms
\begin{equation*}
\big\|f\big\|_{\dot{H}^s_x(\R^d)}:= \big\|
|\xi|^s\hat{f}\big\|_{L^2_x(\R^d)},~~\big\|f\big\|_{{H}^s_x(\R^d)}:=
\big\| \langle\xi\rangle^s\hat{f}\big\|_{L^2_x(\R^d)}.
\end{equation*}

We will also need the Littlewood-Paley projection operators.
Specifically, let $\varphi(\xi)$ be a smooth bump function adapted
to the ball $|\xi|\leq 2$ which equals 1 on the ball $|\xi|\leq 1$.
For each dyadic number $N\in 2^{\mathbb{Z}}$, we define the
Littlewood-Paley operators
\begin{equation*}
\aligned \widehat{P_{\leq N}f}(\xi)& :=
\varphi\Big(\frac{\xi}{N}\Big)\widehat{f}(\xi), \\
\widehat{P_{> N}f}(\xi)& :=
\Big(1-\varphi\Big(\frac{\xi}{N}\Big)\Big)\widehat{f}(\xi), \\
\widehat{P_{N}f}(\xi)& :=
\Big(\varphi\Big(\frac{\xi}{N}\Big)-\varphi\Big(\frac{2\xi}{N}\Big)\Big)\widehat{f}(\xi).
\endaligned
\end{equation*}
Similarly we can define $P_{<N}$, $P_{\geq N}$, and $P_{M<\cdot\leq
N}=P_{\leq N}-P_{\leq M}$, whenever $M$ and $N$ are dyadic numbers.
We will frequently write $f_{\leq N}$ for $P_{\leq N}f$ and
similarly for the other operators.

The Littlewood-Paley operators commute with derivative operators,
the free propagator, and the conjugation operation. They are
self-adjoint and bounded on every $L^p_x$ and $\dot{H}^s_x$ space
for $1\leq p\leq \infty$ and $s\geq 0$, moreover, they also obey the
following
 Bernstein estimates
\begin{eqnarray*}\label{bernstein}
 \big\| P_{\geq N} f \big\|_{L^p} & \lesssim & N^{-s} \big\|
|\nabla|^{s}P_{\geq N} f \big\|_{L^p}, \\
\big\||\nabla|^s P_{\leq N} f \big\|_{L^p} & \lesssim  & N^{s}
\big\|
P_{\leq N} f \big\|_{L^p},  \\
\big\||\nabla|^{\pm s} P_{N} f \big\|_{L^p} & \thicksim & N^{\pm s}
\big\|
P_{N} f \big\|_{L^p},  \\
\big\| P_{\leq N} f \big\|_{L^q} & \lesssim &
N^{\frac{d}{p}-\frac{d}{q}} \big\|
P_{\leq N} f \big\|_{L^p},  \\
\big\| P_{ N} f \big\|_{L^q} & \lesssim &
N^{\frac{d}{p}-\frac{d}{q}} \big\|P_{ N} f \big\|_{L^p},
\end{eqnarray*}
where  $s\geq 0$ and $1\leq p\leq q \leq \infty$.

We will also use the following basic inequalities.
\begin{lemma}[\cite{MR05annmath}]\label{lem:nonltermest}
For any $z\in\C$ with $z=z_1+iz_2$, there holds
\begin{align}\nonumber
&\big|(1+z_1)|1+z|^\frac4d-1-\big(\tfrac4d+1\big)z_1+iz_2\big(|1+z|^\frac4d-1\big)\big|\\\label{equ:z1est}\leq&\begin{cases}
C\big(|z|^{1+\frac43}+|z|^2\big)\quad\text{if}\quad d=3\\
C|z|^2\quad\text{if}\quad d\geq4,
\end{cases}
\end{align}
and
\begin{align}\nonumber
&\big|(1+z_1)|1+z|^\frac4d-1-\big(\tfrac4d+1\big)z_1-\tfrac2d\big(\tfrac4d+1\big)z_1^2-\tfrac2dz_2^2\big|
\\\label{equ:z2est}\leq&\begin{cases}C|z|^3\quad\text{if}\quad d=3\\
C|z|^{2+\frac2d}\quad\text{if}\quad d\geq4,
\end{cases}
\end{align}
and
\begin{align}\nonumber
&\big||1+z|^{2+\frac4d}-1-\big(\tfrac4d+2\big)z_1-\big(\tfrac2d+1\big)\big(\tfrac4d+1\big)z_1^2-\big(\tfrac2d+1\big)z_2^2\big|
\\\label{equ:z3est}
\leq&\begin{cases}
C\big(|z|^{2+\frac43}+|z|^3\big)\quad\text{if}\quad d=3\\
C|z|^3\quad\text{if}\quad d\geq4.
\end{cases}
\end{align}

\end{lemma}

\subsection{Nonlinear estimate}

For $N>1$, we define the Fourier multiplier $I_N$ given  by
$$\widehat{I_Nu}(\xi):=m_N(\xi)\hat{u}(\xi),$$
where $m_N(\xi)$ is a smooth radial decreasing cut off function by
\eqref{mxidy}. Let us collect basic properties of $I_N$.

\begin{lemma}[\cite{VZ07}]\label{highcont} Let $1<p<\infty$ and $0\leq\sigma\leq
s<1$. Then,
\begin{align}\label{equ2.1}
\|I_Nf\|_{L^p}\lesssim&\|f\|_{L^p},\\\label{equ2.2}
\big\||\nabla|^\sigma
P_{>N}f\big\|_{L^p}\lesssim&N^{\sigma-1}\big\|\nabla
I_Nf\big\|_{L^p},\\\label{equ2.3}
\|f\|_{H^s}\lesssim\|I_Nf\|_{H^1}\lesssim&N^{1-s}\|f\|_{H^s}.
\end{align}
\end{lemma}

We will need the following fractional calculus estimates from \cite{CW}.

\begin{lemma}[Fractional product rule \cite{CW}]
Let $s\geq0$, and $1<r,r_j,q_j<\infty$ satisfy
$\frac1r=\frac1{r_i}+\frac1{q_i}$ for $i=1,2$. Then
\begin{equation}\label{moser}
\big\||\nabla|^s(fg)\big\|_{L_x^r(\R^d)}\lesssim\|f\|_{{L_x^{r_1}(\R^d)}}\big\||\nabla|^sg
\big\|_{{L_x^{q_1}(\R^d)}}+\big\||\nabla|^sf\big\|_{{L_x^{r_2}(\R^d)}}\|g\|_{{L_x^{q_2}(\R^d)}}.
\end{equation}
From \cite{CR09}, we have
\begin{equation}\label{moser2}
\big\|I_N\nabla(fg)\big\|_{L_x^r(\R^d)}\lesssim\|f\|_{{L_x^{r_1}(\R^d)}}\big\|I_N\nabla g
\big\|_{{L_x^{q_1}(\R^d)}}+\big\|I_N\nabla f\big\|_{{L_x^{r_2}(\R^d)}}\|g\|_{{L_x^{q_2}(\R^d)}}.
\end{equation}
\end{lemma}

\begin{lemma}[Fractional chain rule \cite{CW}]
Let $G\in
C^1(\mathbb{C}),~s\in(0,1],$ and $1<r,r_1,r_2<+\infty$ satisfy $\frac1r=\frac1{r_1}+\frac1{r_2}.$
Then
\begin{equation}\label{fraclsfz}
\big\||\nabla|^s
G(u)\big\|_r\lesssim\|G'(u)\|_{r_1}\big\||\nabla|^su\big\|_{r_2}.
\end{equation}
\end{lemma}

As noted in the introduction, one needs to estimate the commutator
 $|Iu|^pIu-I(|u|^pu)$ in the increment of modified
energy $E(Iu)(t)$. When $p$ is an even integer, one can use
multilinear analysis to  expand this commutator into a product of
Fourier transforms of $u$ and $Iu$,  and carefully measure
frequency interactions to derive an estimate (see for example
\cite{CGT09}). However, this is not possible when $p$ in not an even
integer. Instead, Visan and Zhang in \cite{VZ07} established the following
rougher (weaker, but more robust)  estimate:

\begin{lemma}[commutator  estimate, \cite{VZ07}]\label{commutor}
Let $1<r,r_1,r_2<\infty$ be such that
$\frac1r=\frac1{r_1}+\frac1{r_2}$ and let $0<\nu<s.$ Then,
\begin{equation}\label{equ2.4}
\big\|I_N(fg)-(I_Nf)g\big\|_{L^r}\lesssim
N^{-(1-s+\nu)}\|I_Nf\|_{L^{r_1}}\big\|\langle\nabla\rangle^{1-s+\nu}g\big\|_{L^{r_2}}.
\end{equation}
Furthermore, let $I$ be a time interval and let $\frac1{1+\min\{1,\frac4d\}}<s<1,$  then we have
\begin{align}\label{equ:inuf}
\big\|\nabla I_N(|u|^\frac4du)-(I_N\nabla u)|u|^\frac4d\big\|_{L_t^2L_x^\frac{2d}{d+2}(I\times\R^d)}\lesssim N^{-\min\{1,\frac4d\}s_+}\|\langle\nabla\rangle I_Nu\|_{S^0(I)}^{1+\frac4d},\\\label{equ:nainfu}
\big\|\langle\nabla\rangle I_N(|u|^\frac4du)\big\|_{N^0(I)}\lesssim(|I|^\frac{2s}d+N^{-\min\{1,\frac4d\}s_+})\big\|\langle\nabla\rangle I_Nu\big\|_{S^0(I)}^{1+\frac4d},\\\label{equ:nonlestque}
\big\|\langle\nabla\rangle^{\min\{1,\frac4d\}s_-}(|u|^\frac4d)\big\|_{L_t^\infty L_x^\frac{d}2}\lesssim\big\|\langle\nabla\rangle I_Nu\big\|_{L_t^\infty L_x^2}^\frac4d,
\end{align}
where $S^0(I)$ and $N^0(I)$ is defined in Definition \ref{def1} below.
\end{lemma}

\begin{remark}
It is easy to check that for $(q,r)\in\Lambda_0$,
\begin{equation}
\big\|\langle\nabla\rangle^su\big\|_{L_t^qL_x^r}\lesssim\big\|\langle\nabla\rangle I_Nu\big\|_{L_t^qL_x^r},
\end{equation}
where $\Lambda_0$ is defined in Definition \ref{def1} below.
\end{remark}

 \subsection{Strichartz estimates and local well-posedness}\label{sze}

In this subsection, we consider the Cauchy problem
\begin{equation} \label{equ2}
    \left\{ \aligned &iu_t+\Delta u-f(u)=  0, \\
    &u(0)=u_0.
    \endaligned
    \right.
\end{equation}
The integral equation for the Cauchy problem $(\ref{equ2})$ can be
written as
\begin{equation}\label{inte1}
u(t,x)=e^{i(t-t_0)\Delta}u(t_0)-i\int_{t_0}^te^{i(t-s)\Delta}f(u(s))ds.
\end{equation}

Now we recall the dispersive estimate for the free Schr\"odinger
operator $U(t)=e^{it\Delta}$. From the explicit formula
$$e^{it\Delta}f(x)=\frac1{(4i\pi
t)^\frac{d}{2}}\int_{\R^d}e^{i\frac{|x-y|^2}{4t}}f(y)dy,$$ it is easy to get the
standard dispersive inequality
\begin{equation}\label{disper}
\big\|e^{it\Delta}f
\big\|_{L^\infty_x(\R^d)}\lesssim|t|^{-\frac{d}2}
\|f\|_{L^1_x(\R^d)}
\end{equation}
for all $t\neq0.$ On the other hand, since the free operator
conserves the $L_x^2(\R^d)$-norm, we obtain by interpolation
\begin{equation}\label{dispers}
\big\|e^{it\Delta}f \big\|_{L^q_x(\R^d)} \leq
C|t|^{-d(\frac12-\frac1q)} \|f\|_{L^{q'}_x(\R^d)}
\end{equation}
for all $t\neq0$ and $2\leq q\leq+\infty,~\frac1q+\frac1{q'}=1$.

The Strichartz estimates involve the following definitions:
\begin{definition}\label{def1}
A pair of Lebesgue space exponents $(q,r)$ are called
Schr\"{o}dinger admissible for $\mathbb{R}^{d+1}$, or denote by
$(q,r)\in \Lambda_{0}$ when $q,r\geq 2,~(q,r,d)\neq (2,\infty,2)$,
and
\begin{equation}\label{equ21}
\frac{2}{q}=d\Big(\frac{1}{2}-\frac{1}{r}\Big).
\end{equation}
For a fixed spacetime slab $I\times\R^d$, we define the Strichartz
norm
$$\|u\|_{S^0(I)}:=\sup_{(q,r)\in\Lambda_0}\|u\|_{L_t^q
L_x^r(I\times\R^d)},~d\geq3$$ We denote $S^0(I)$ to be the closure of all
test functions under this norm and write $N^0(I)$ for the dual of
$S^0(I)$.
\end{definition}

According to the above dispersive estimate, the abstract duality and
interpolation argument(see \cite{KeT98}), we have the following
Strichartz estimates.
\begin{lemma}[Strichartz estimate, \cite{GiV85b,KeT98}]\label{lem22}
Let $s\geq0,$ and let $I$ be a compact time interval, and let $u:
I\times\R^d\to\C$ be a solution of the Schr\"odinger equation
$$iu_t+\Delta u+h=0.$$
Then, for all $t_0\in I$
$$\big\||\nabla|^su\big\|_{S^0(I)}\leq
C\big\||\nabla|^su(t_0)\big\|_{L_x^2(\R^d)}+\big\||\nabla|^sh\big\|_{N^0(I)}.$$
\end{lemma}

By the fixed point argument, we have the following local well-posedness(LWP).

\begin{lemma}[$H^s$-LWP]\label{lem:hslwp}
Let $s\in(0,1]$, $u_0\in H^s(\R^d)$ and
\begin{equation}\label{equ:tlwp}
T_{\rm LWP}=c\|u_0\|_{H^s}^{-\frac2s}
\end{equation}
with $c$ small depending the constant in Strichartz estimate and Sobolev embedding. Then, there exists a unique solution $u(t)$ to \eqref{equ:nls} on $[0,T_{\rm LWP}]$ and satisfying
\begin{align}\label{equ:uest}
\|u\|_{S^0([0,T_{\rm LWP}])}\leq 2C\|u_0\|_{L_x^2},&~\big\|\langle \nabla\rangle^s u\big\||_{S^0([0,T_{\rm LWP}])}\leq 2C\|u_0\|_{H^s}.
\end{align}
\end{lemma}

\begin{proof}
  We apply the Banach fixed point argument to prove this lemma. First,
we define the map
  $$\Phi(u)=e^{it\Delta}u_0-i\int_0^t e^{i(t-s)\Delta}(|u|^\frac4du)(s)\;ds$$
  on the complete metric space $B$
  \begin{equation*}
  \begin{split}
  B:=\big\{u\in C(I;H^s):~&\|u\|_{S^0([0,T_{\rm LWP}])}\leq 2C\|u_0\|_{L_x^2},~\big\|\langle \nabla\rangle^s u\big\||_{S^0([0,T_{\rm LWP}])}\leq 2C\|u_0\|_{H^s}\big\}
  \end{split}
  \end{equation*}
with the metric
$d(u,v)=\big\|u-v\big\|_{L_{t,x}^{\frac{2(d+2)}{d}}([0,T_{\rm
LWP}]\times\R^d)}$,  where $C$ is the constant in Strichartz
estimates.

It suffices to prove that the operator $\Phi(u)$  is a contraction map on $B$ for $[0,T_{\rm LWP}]$. In fact, if
$u\in B$, then  by Strichartz estimate, H\"older's inequality and Sobolev embedding, we have
  \begin{align*}
  \|\Phi(u)\|_{S^0([0,T_{\rm LWP}])}\leq&C\|u_0\|_{L_x^2}+C\big\||u|^\frac4du\big\|_{L_{t,x}^{\frac{2(d+2)}{d+4}}([0,T_{\rm LWP}]\times\R^d)}\\
 \leq&C\|u_0\|_{L_x^2}+C\|u\|_{L_{t,x}^{\frac{2(d+2)}{d}}([0,T_{\rm LWP}]\times\R^d)}\|u\|_{L_{t,x}^{\frac{2(d+2)}{d}}([0,T_{\rm LWP}]\times\R^d)}^\frac4d \\
  \leq&C\|u_0\|_{L_x^2}+C\|u\|_{L_{t,x}^{\frac{2(d+2)}{d}}([0,T_{\rm LWP}]\times\R^d)}\Big(T_{\rm LWP}^\frac{s}2\big\|\langle \nabla\rangle^s u\big\||_{S^0([0,T_{\rm LWP}])}\Big)^\frac4d\\
  \leq&C\|u_0\|_{L_x^2}+2^\frac4dC^{2+\frac4d}c^\frac4d\|u_0\|_{L_x^2}\\
  \leq&2C\|u_0\|_{L_x^2}
  \end{align*}
  provided that $2^\frac4dC^{2+\frac4d}c^\frac4d<1$ with $T_{\rm LWP}=c\|u_0\|_{H^s}^{-\frac2s}$. Similarly, we obtain
  \begin{align*}
  \big\|\langle \nabla\rangle^s \Phi(u)\big\||_{S^0([0,T_{\rm LWP}])}\leq &C\|u_0\|_{H^s}+C\big\|\langle \nabla\rangle^s(|u|^\frac4du)
  \big\|_{L_{t,x}^{\frac{2(d+2)}{d+4}}([0,T_{\rm LWP}]\times\R^d)}\\
  \leq &C\|u_0\|_{H^s}+C\|\langle \nabla\rangle^su\|_{L_{t,x}^{\frac{2(d+2)}{d}}([0,T_{\rm LWP}]\times\R^d)}\Big(T_{\rm LWP}^\frac{s}2\big\|\langle \nabla\rangle^s u\big\||_{S^0([0,T_{\rm LWP}])}\Big)^\frac4d\\
  \leq&C\|u_0\|_{H^s}+2^\frac4dC^{2+\frac4d}c^\frac4d\|u_0\|_{H^s}\leq2C\|u_0\|_{L_x^2}.
  \end{align*}
  Hence, $\Phi(u)\in B.$

On the other hand,  for $u, v\in B$, by Strichartz
estimate, we obtain
  \begin{align*}
  d(\Phi(u),\Phi(v))=&\big\|\Phi(u)-\Phi(v)\big\|_{L_{t,x}^{\frac{2(d+2)}{d}}([0,T_{\rm LWP}]\times\R^d)}\\
  \leq&C\big\||u|^\frac4du-|v|^\frac4dv\big\|_{L_{t,x}^{\frac{2(d+2)}{d+4}}([0,T_{\rm LWP}]\times\R^d)}\\
  \leq&C\|u-v\|_{L_{t,x}^{\frac{2(d+2)}{d}}([0,T_{\rm LWP}]\times\R^d)}\|(u,v)\|_{L_{t,x}^{\frac{2(d+2)}{d}}([0,T_{\rm LWP}]\times\R^d)}^\frac4d\\
  \leq&C\|u-v\|_{L_{t,x}^{\frac{2(d+2)}{d}}([0,T_{\rm LWP}]\times\R^d)}\Big(2T_{\rm LWP}^\frac{s}2\big\|\langle \nabla\rangle^s u\big\||_{S^0([0,T_{\rm LWP}])}\Big)^\frac4d\\
  \leq&2^\frac8dC^{2+\frac4d}c^\frac4d d(u,v)\\
  \leq&\frac12d(u,v),
  \end{align*}
  provided that $2^\frac8dC^{2+\frac4d}c^\frac4d<\tfrac12.$

  Therefore,  applying the fixed point theorem gives a unique solution
$u$ of \eqref{equ:nls} on $[0,T_{\rm LWP}]$ which satisfies the bound
\eqref{equ:uest}.

  Therefore,  applying the fixed point theorem gives a unique solution
$u$ of \eqref{equ:nls} on $[0,T_{\rm LWP}]$ which satisfies the bound
\eqref{equ:uest}.

\end{proof}

\begin{corollary}[`Modified' $H^s$-LWP]\label{cor:modlwp}
For $T^\ast\geq T_{\rm LWP}$, we denote
\begin{equation}\label{equ:mtlwp}
\tilde{T}_{\rm LWP}:=c_0\big\|\langle\nabla\rangle I_{N(T^\ast)}u_0\|_{L^2}^{-\frac2s}
\end{equation}
with $c_0$ small. By \eqref{equ2.3}, we obtain
$$\tilde{T}_{\rm LWP}\leq c_0\|u_0\|_{H^s}^{-\frac2s}.$$
This together with Lemma \ref{lem:hslwp} implies that \eqref{equ:nls} is well-posedness on $[0,\tilde{T}_{\rm LWP}]$, and
\begin{align}\label{equ:uest123}
\|u\|_{S^0([0,\tilde T_{\rm LWP}])}\leq 2C\|u_0\|_{L_x^2},&~\big\|\langle \nabla\rangle^s u\big\||_{S^0([0,\tilde T_{\rm LWP}])}\leq 2C\|u_0\|_{H^s}.
\end{align}
Moreover, if
 $\frac{1}{1+\min\{1,\frac4d\}}<s<1,~~\frac{4s}{\min\{4,d\}s-4(1-s)}<\frac1{\beta}$, then there holds
\begin{equation}\label{equ:mlwpes}
\big\|\langle\nabla\rangle I_Nu\big\|_{S^0([0,\tilde{T}_{\rm LWP}])}\lesssim\|I_Nu\|_{H^1}.
\end{equation}
\end{corollary}

\begin{proof}
First, we have by Strichartz estimate and \eqref{equ:nainfu}
\begin{align}\nonumber
\big\|\langle\nabla\rangle I_Nu\big\|_{S^0([0,\tilde{T}_{\rm LWP}])}\leq&C\big\|I_Nu_0\big\|_{H^1}+C\big\|\langle\nabla\rangle I_N(|u|^\frac4du)\big\|_{N^0([0,\tilde{T}_{\rm LWP}])}\\\nonumber
\leq&C\big\|I_Nu_0\big\|_{H^1}+C\tilde{T}_{\rm LWP}^\frac{2s}d\big\|\langle\nabla\rangle I_Nu\big\|_{S^0([0,\tilde{T}_{\rm LWP}])}^{1+\frac4d}\\\label{equ:cnt}
&+CN(T^\ast)^{-\min\{1,\frac4d\}s+\var}\big\|\langle\nabla\rangle I_Nu\big\|_{S^0([0,\tilde{T}_{\rm LWP}])}^{1+\frac4d}
\end{align}
for any $\var>0$ sufficiently small. Using \eqref{equ:ntass}, monotonicity \eqref{equ:almons}, Remark \ref{rem:uths}: \eqref{equ:uthslam} and \eqref{equ2.3}, we get
\begin{align*}
N(T^\ast)^{-1}=&\lambda(T^\ast)^{\frac1\beta}\lesssim\lambda(0)^{\frac1\beta}
\lesssim\|u_0\|_{H^s}^{-\frac1{s\beta}}\\
\lesssim&\big(N(T^\ast)^{s-1}\|I_{N(T^\ast)}u_0\|_{H^1}\big)^{-\frac1{s\beta}}
\end{align*}
and so
$$N(T^\ast)^{-1}\lesssim \|I_{N(T^\ast)}u_0\|_{H^1}^{-\frac1{s\beta+1-s}}.$$
Hence, from $\frac{4s}{\min\{4,d\}s-4(1-s)}<\frac1{\beta}$, we know $\frac{4(s\beta+1-s)}{d}-\min\{1,\frac4d\}s<0$ and
\begin{align*}
N(T^\ast)^{-\min\{1,\frac4d\}s+\var}=&N(T^\ast)^{-\min\{1,\frac4d\}s+\var+\frac{4(s\beta+1-s)}{d}}N(T^\ast)^{-\frac{4(s\beta+1-s)}{d}}\\
\lesssim&\|I_{N(T^\ast)}u_0\|_{H^1}^{-\frac4d}\sim \tilde{T}_{\rm LWP}^\frac{2s}d.
\end{align*}
Plugging this into \eqref{equ:cnt} implies
$$\big\|\langle\nabla\rangle I_Nu\big\|_{S^0([0,\tilde{T}_{\rm LWP}])}\leq C\big\|I_Nu_0\big\|_{H^1}+C\tilde{T}_{\rm LWP}^\frac{2s}d\big\|\langle\nabla\rangle I_Nu\big\|_{S^0([0,\tilde{T}_{\rm LWP}])}^{1+\frac4d}.$$
Therefore, \eqref{equ:mlwpes} follows from standard continuous argument.

\end{proof}

\begin{remark}
Here the restriction on $s$ is different from \cite{VZ07}.
\end{remark}

As a direct application of $H^s$-LWP, we can control the number of LWP intervals covering the interval $[t_k,t_{k+1}]$ as follows.

\begin{lemma} Let  $\frac{1}{1+\min\{1,\frac4d\}}<s<1,~~\frac{4s}{\min\{4,d\}s-4(1-s)}<\frac1{\beta}$.
Let $\{t_k\}_{k_0\leq k\leq k^+}$ be defined as in \eqref{equ:ltk}, and $T^\ast\geq t_{k+1}$. We cover the interval $[t_k,t_{k+1}]$ by LWP time interval $\{\tau_{k}^j\}_{1\leq j\leq J_k}$ given by Corollary \ref{cor:modlwp}. Then, we have
\begin{equation}\label{equ:jk}
J_k\lesssim kN(T^\ast)^\frac{2(1-s)}s.
\end{equation}
\end{lemma}

\begin{proof}
First, it follows from \eqref{equ:tlwp} that
\begin{equation}\label{equ:taukk-1}
\tau_k^{j+1}-\tau_k^j\sim\frac1{\big\|\langle\nabla\rangle I_{N(T^\ast)}u(\tau_k^j)\|_{L^2}^\frac2s}\gtrsim\Big(\frac1{N(T^\ast)^{1-s}\|u(\tau_k^j)\|_{H^s}}\Big)^\frac2s.
\end{equation}
This together with Remark \ref{rem:uths}: \eqref{equ:uthslam} and the almost monotonicity \eqref{equ:almons} implies
\begin{equation}\label{equ:taukktao}
\tau_k^{j+1}-\tau_k^j\gtrsim\frac1{N(T^\ast)^\frac{2(1-s)}{s}}
\lambda(\tau_k^j)^2\sim\frac1{N(T^\ast)^\frac{2(1-s)}{s}}\lambda(t_k)^2.
\end{equation}
And so \eqref{equ:jk} follows from the control of the blowup speed \eqref{equ:dotime}.

\end{proof}

\begin{lemma}\label{lem:xianchafa}
 Let $[\tau_k^j,\tau_k^{j+1}]$ be a LWP time interval as given by  Corollary \ref{cor:modlwp}. Then,
there holds:
\begin{align}\label{equ:xianchafa}
&\Big\|I_{N(t)}(|u|^\frac4du)-I_{N(t)}u|I_{N(t)}u|^\frac4d\Big\|_{L_t^2L_x^\frac{2d}{d+2}([\tau_k^j,\tau_k^{j+1}]\times\R^d)}\\\nonumber
\lesssim& \lambda(t_k)^{\frac1\beta\min\{1,\frac4d\}s_-}\big\|\langle\nabla\rangle I_Nu\big\|_{L_t^\infty L_x^2([\tau_k^j,\tau_k^{j+1}]\times\R^d)}^\frac4d.
\end{align}

\end{lemma}

\begin{proof}
 we have by triangle inequality
\begin{align*}
&\big\||I_Nu|^\frac4dI_Nu-I_N(|u|^\frac4du)\big\|_{L_t^2L_x^\frac{2d}{d+2}([\tau_k^j,\tau_k^{j+1}]\times\R^d)}\\
\lesssim&\big\|I_Nu\big(|I_Nu|^\frac4d-|u|^\frac4d\big)\big\|_{L_t^2L_x^\frac{2d}{d+2}([\tau_k^j,\tau_k^{j+1}]\times\R^d)}+\big\|(I_Nu)|u|^\frac4d-I_N(|u|^\frac4du)\big\|_{L_t^2L_x^\frac{2d}{d+2}([\tau_k^j,\tau_k^{j+1}]\times\R^d)}.
\end{align*}
Using \eqref{equ2.4} with $\nu=\min\{1,\frac4d\}s-(1-s)_-$ and \eqref{equ:nonlestque}, we estimate
\begin{align*}
&\big\|(I_Nu)|u|^\frac4d-I_N(|u|^\frac4du)\big\|_{L_t^2L_x^\frac{2d}{d+2}([\tau_k^j,\tau_k^{j+1}]\times\R^d)}\\
\lesssim&N(t_k)^{-\min\{1,\frac4d\}s_+}\|I_Nu\|_{L_t^2L_x^\frac{2d}{d-2}([\tau_k^j,\tau_k^{j+1}]\times\R^d)}\big\|\langle\nabla\rangle^{\min\{1,\frac4d\}s_-}(|u|^\frac4d)\big\|_{L_t^\infty L_x^\frac{d}2([\tau_k^j,\tau_k^{j+1}]\times\R^d)}\\
\lesssim&N(t_k)^{-\min\{1,\frac4d\}s_+}\|u\|_{L_t^\infty L_x^2}
\big\|\langle\nabla\rangle I_Nu\big\|_{L_t^\infty L_x^2([\tau_k^j,\tau_k^{j+1}]\times\R^d)}^\frac4d\\
\lesssim&\lambda(t_k)^{\frac1\beta\min\{1,\frac4d\}s_-}\big\|\langle\nabla\rangle I_Nu\big\|_{L_t^\infty L_x^2([\tau_k^j,\tau_k^{j+1}]\times\R^d)}^\frac4d.
\end{align*}
Similarly,
\begin{align*}
&\big\|I_Nu\big(|I_Nu|^\frac4d-|u|^\frac4d\big)\big\|_{L_t^2L_x^\frac{2d}{d+2}([\tau_k^j,\tau_k^{j+1}]\times\R^d)}\\
\lesssim&\|I_Nu\|_{L_t^2L_x^\frac{2d}{d-2}([\tau_k^j,\tau_k^{j+1}]\times\R^d)}\big\||I_Nu-u|^{\min\{1,\frac4d\}}(|I_Nu|+|u|)^{\frac4d-\min\{1,\frac4d\}}\big\|_{L_t^\infty L_x^\frac{d}2([\tau_k^j,\tau_k^{j+1}]\times\R^d)}\\
\lesssim&\|u\|_{L_t^\infty L_x^2}^{1+\frac4d-\min\{1,\frac4d\}}\big\|I_Nu-u\big\|_{L_t^\infty L_x^2([\tau_k^j,\tau_k^{j+1}]\times\R^d)}^{\min\{1,\frac4d\}}\\
\lesssim&   N(t_k)^{-\min\{1,\frac4d\}s}\big\|\langle\nabla\rangle I_Nu\big\|_{L_t^\infty L_x^2([\tau_k^j,\tau_k^{j+1}]\times\R^d)}^{\min\{1,\frac4d\}}\|u\|_{L_t^\infty L_x^2}^{1+\frac4d-\min\{1,\frac4d\}}\\
\lesssim&\lambda(t_k)^{\frac1\beta\min\{1,\frac4d\}s}\big\|\langle\nabla\rangle I_Nu\big\|_{L_t^\infty L_x^2([\tau_k^j,\tau_k^{j+1}]\times\R^d)}^{\frac4d}.
\end{align*}

\end{proof}

\subsection{Almost conservation law}

First, we consider the increments in the local-wellposedness time interval.

\begin{lemma}[Control of increments in LWP time interval]\label{lem:conincre} Let $\frac{1}{1+\min\{1,\frac4d\}}<s<1,~\frac{4s}{\min\{4,d\}s-4(1-s)}<\frac1{\beta}$.  Let $\{\tau_{k}^j\}_{1\leq j\leq J_k}$  be the LWP time interval given by Corollary \ref{cor:modlwp}. Then, for any $T^\ast\geq\tau_k^{j+1}$,
the modified energy has a slow increment over the LWP time interval:
\begin{equation}\label{equ:modenerlwp}
\begin{split}
&\big|E(I_{N(T^\ast)}u(\tau_{k}^{j+1}))-E(I_{N(T^\ast)}u(\tau_{k}^{j}))\big|\\ \leq& C N(T^\ast)^{-\min\{1,\frac4d\}s_+}\Big(\|I_{N(T^\ast)}u(\tau_k^j)\|_{H^1}^{2+\frac4d}+\|I_{N(T^\ast)}u(\tau_k^j)\|_{H^1}^{2+\frac8d}\Big).
\end{split}
\end{equation}Here, the constant $C$ depends only on $s$.
The modified momentum has a slow increment over the LWP time interval:
\begin{equation}\label{equ:modenerlwp123}
\big|P(I_{N(T^\ast)}u(\tau_{k}^{j+1}))-P(I_{N(T^\ast)}u(\tau_{k}^{j}))\big|\leq CN(T^\ast)^{-\min\{1,\frac4d\}s_+}\|I_{N(T^\ast)}u(\tau_k^j)\|_{H^1}^{2+\frac 4d-\frac1s}.
\end{equation}
Moreover,
\begin{equation}\label{equ:dius}
\|I_{N(T^\ast)}u(\tau_k^j)\|_{H^1}\leq\Big(\frac{N(T^\ast)}{N(\tau_k^j)}\Big)^{1-s}\frac1{\lambda(\tau_k^j)}.
\end{equation}

\end{lemma}

\begin{proof}
First, \eqref{equ:modenerlwp} follows from Lemma 4.2 in \cite{VZ07} and Corollary \ref{cor:modlwp}. And \eqref{equ:dius} follows from (3.20) in \cite{CR09}. Thus, we only need to show \eqref{equ:modenerlwp123}. We write $N=N(T^\ast)$ in this proof.
Note that
$${\rm Re}\int\nabla\Delta I_Nu\overline{I_Nu}=0\quad\text{and}\quad {\rm Re}\int \nabla(|I_Nu|^\frac4dI_Nu)\overline{I_Nu}=0,$$
we derive that
\begin{align*}
\frac{d}{dt}P(I_Nu)=&{\rm Im}\int \nabla\pa_tI_Nu\overline{I_Nu}+{\rm Im}\int \nabla I_Nu\overline{\pa_tI_Nu}\\
=&{\rm Re}\int\big(\nabla\Delta I_Nu+\nabla I_N(|u|^\frac4du)\big)\overline{I_Nu}-{\rm Re}\int\nabla I_Nu\big(\Delta\overline{I_Nu}+\overline{I_N(|u|^\frac4du)}\big)\\
=&2{\rm Re}\int\nabla I_N(|u|^\frac4du)\overline{I_Nu}\\
=&2{\rm Re}\int\nabla (I_N(|u|^\frac4du)-|I_Nu|^\frac4dI_Nu)\overline{I_Nu}.
\end{align*}
Hence,
\begin{align*}
&\big|P(I_{N(T^\ast)}u(\tau_{k}^{j+1}))-P(I_{N(T^\ast)}u(\tau_{k}^{j}))\big|\\
\leq&\Big|\int_{\tau_k^j}^{\tau_k^{j+1}} \frac{d}{dt}P(I_Nu)\;dt\Big|\\
\lesssim&\int_{\tau_k^j}^{\tau_k^{j+1}}\int_{\R^d} \big|\nabla (I_N(|u|^\frac4du)-|I_Nu|^\frac4dI_Nu)\overline{I_Nu}\big|\;dx\;dt\\
\lesssim&\|\nabla[I_N(|u|^\frac4du)-|I_Nu|^\frac4dI_Nu]\|_{L_t^2L_x^\frac{2d}{d+2}([\tau_k^j,\tau_k^{j+1}]\times\R^d)}
\|I_Nu\|_{L_t^2L_x^\frac{2d}{d-2}([\tau_k^j,\tau_k^{j+1}]\times\R^d)}.
\end{align*}
 It follows from (4.5) in \cite{VZ07} that
\begin{align*}
\big\|\nabla\big[I_N(|u|^\frac4du)-|I_Nu|^\frac4dI_Nu\big]\big\|_{L_t^2L_x^\frac{2d}{d+2}([\tau_k^j,\tau_k^{j+1}]\times\R^d)}\lesssim
N^{-\min\{1,\frac 4d\}s_+}\|\langle\nabla\rangle I_N
u\|_{S^0([\tau_k^j,\tau_k^{j+1}])}^{1+\frac 4d},
\end{align*}
this together with H\"older's inequality implies that
\begin{align*}
&\big|P(I_{N(T^\ast)}u(\tau_{k}^{j+1}))-P(I_{N(T^\ast)}u(\tau_{k}^{j}))\big|\\
\lesssim&N^{-\min\{1,\frac 4d\}s_+}\|\langle\nabla\rangle I_N u\|_{S^0([\tau_k^j,\tau_k^{j+1}])}^{2+\frac 4d}|\tau_k^j-\tau_k^{j+1}|^\frac12\\
\lesssim&N^{-\min\{1,\frac 4d\}s_+}\|\langle\nabla\rangle I_N
u(\tau_k^j)\|_{L^2}^{2+\frac 4d-\frac1s}.
\end{align*}

\end{proof}

Next, we consider the initial data.

\begin{lemma}\label{lem:einit}
Let
\begin{equation}\label{equ:thetadef}
\Xi(t):=\frac{\lambda(t)^2}2\int\big(|\nabla G(0)|^2-|\nabla I_{N(t)}G(0)|^2\big)\;dx.
\end{equation}
Then, we have for $t\in[0,T^+],$
\begin{align}\label{equ:eint0}
\Big|E(I_{N(t)}u(0))+\frac{\Xi(t)}{\lambda(t)^2}\Big|\lesssim&N(t)^{2(1-s)}+\frac1{\lambda(t)^{2-\frac{1-\beta}{\beta}\frac2{d+2}}},\\\label{equ:pu0}
\big|P(I_{N(t)}u(0))\big|\lesssim&N(t)^{1-s}+\frac1{\lambda(t)^{1-\frac{1-\beta}{2\beta}}}.
\end{align}

\end{lemma}

\begin{proof}
Note that $u_0=G(0)+H(0)$, we have
\begin{align}\label{equ:etnut}
\Big|E(I_{N(t)}u(0))+\frac{\Xi(t)}{\lambda(t)^2}\Big|\lesssim& \Big|E(I_{N(t)}G(0))+\frac{\Xi(t)}{\lambda(t)^2}\Big|\\\label{equ:einut2}
&+\big|E(I_{N(t)}\big(G(0)+H(0)\big)-E(I_{N(t)}G(0))\big|.
\end{align}

{\bf The estimate of \eqref{equ:etnut}:} A simple computation shows
\begin{equation}
E(I_{N(t)}u(0))+\frac{\Xi(t)}{\lambda(t)^2}=E(G(0))+\frac{d}{2(d+2)}\int\big(|G(0)|^\frac{2(d+2)}{d}-|I_{N(t)}G(0)|^\frac{2(d+2)}{d}\big)\;dx.
\end{equation}
Hence, by $G(0)\in H^1$ and $\|G(0)\|_{\dot{H}^s}\sim\frac1{\lambda(0)^s}$, \eqref{equ:eg0}, \eqref{equ:almons} and \eqref{equ:ntass}, we get
\begin{align}\nonumber
\Big|E(I_{N(t)}G(0))+\frac{\Xi(t)}{\lambda(t)^2}\Big|\leq&|E(G(0))|+\frac{d}{2(d+2)}\int\Big||G(0)|^\frac{2(d+2)}{d}-|I_{N(t)}G(0)|^\frac{2(d+2)}{d}
\Big|\;dx\\\nonumber
\lesssim&\frac1{\sqrt{\lambda(0)}}+\big\|(I_{N(t)}-Id)G(0)\big\|_{L_x^\frac{2(d+2)}{d}}\|G(0)\|_{L_x^\frac{2(d+2)}{d}}^\frac{d+4}d\\\nonumber
\lesssim&\frac1{\sqrt{\lambda(0)}}+\frac{\big\|(I_{N(t)}-Id)G(0)\big\|_{\dot{H}^\frac{d}{d+2}}}{\lambda(0)^\frac{d+4}{d+2}}\\\nonumber
\lesssim&\frac1{\sqrt{\lambda(0)}}+\frac1{\lambda(0)^2}\Big(\frac1{\lambda(0)N(t)}\Big)^\frac{2}{d+2}\\\nonumber
\lesssim&\frac1{\lambda(0)^2}\Big(\lambda(t)^\frac32+\lambda(t)^{\frac{1-\beta}{\beta}\frac2{d+2}}\Big)\\\label{equ:eintu1}
\lesssim&\frac1{\lambda(t)^{2-\frac{1-\beta}{\beta}\frac2{d+2}}}
\end{align}

{\bf The estimate of \eqref{equ:einut2}:} Observe that
\begin{align*}
|E(u+v)-E(u)|\lesssim&\|\nabla u\|_{L_x^2}\|\nabla v\|_{L_x^2}+\|\nabla v\|_{L_x^2}^2+\int\big(|u|^\frac{d+4}d+|v|^\frac{d+4}d\big)|v|\;dx\\
\lesssim&\|\nabla v\|_{L_x^2}^2(1+\|v\|_{L_x^2}^\frac4d)+\|\nabla u\|_{L_x^2}\|\nabla v\|_{L_x^2}\\
&+\|\nabla u\|_{L^2}^\frac{d+4}{d+2}\|\nabla v\|_{L^2}^\frac{d}{d+2}\|u\|_{L^2}^\frac{d+4}{d+2}\|v\|_{L^2}^\frac{d}{d+2}.
\end{align*}
Using \eqref{equcont}, \eqref{equ:g0x}, \eqref{equ:h0hs} and \eqref{equ:g0small}, we obtain
\begin{align}\label{equ:g0uw}
\|\nabla I_{N(t)}G(0)\|_{L^2}\leq\|\nabla G(0)\|_{L^2}\lesssim&\frac1{\lambda(0)},\\\label{equ:h0vt}
\|\nabla I_{N(t)}H(0)\|_{L^2}\lesssim N(t)^{1-s}\|H(0)\|_{H^s}\leq&\lambda(0)^{10}N(t)^{1-s}.
\end{align}
This together with the uniform $L^2$ control \eqref{equ:excmass} implies
\begin{equation}\label{equ:eintgh}
\big|E(I_{N(t)}\big(G(0)+H(0)\big)-E(I_{N(t)}G(0))\big|\lesssim N(t)^{2(1-s)}.
\end{equation}
And so we obtain \eqref{equ:eint0}.

Next, we prove \eqref{equ:pu0}. This part is independent of nonlinear term, so this term is as in \cite{CR09}. In fact, by \eqref{equ:gg0}, we have
\begin{align*}
|P(I_{N(t)}u(0))|\leq&\big|P(I_{N(t)}(G(0)+H(0))-P(I_{N(t)}G(0))\big|\\
&+|P(I_{N(t)}G(0))-P(G(0))\big|+|P(G(0))|\\
\lesssim&\big|P(I_{N(t)}(G(0)+H(0))-P(I_{N(t)}G(0))\big|\\
&+|P(I_{N(t)}G(0))-P(G(0))\big|+\frac1{\sqrt{\lambda(0)}}.
\end{align*}
A simple computation shows that for $u,~v\in\dot{H}^\frac12,$
$$|P(u+v)-P(u)|\lesssim\|v\|_{\dot{H}^\frac12}\big(\|u\|_{\dot{H}^\frac12}+\|v\|_{\dot{H}^\frac12}\big).$$
Combining this with \eqref{equ:g0uw} and \eqref{equ:h0vt}, we derive that
\begin{align*}
&\big|P(I_{N(t)}(G(0)+H(0))-P(I_{N(t)}G(0))\big|\\
\lesssim&\|I_{N(t)}H(0)\|_{\dot{H}^\frac12}\Big(\frac1{\lambda(0)}+\|I_{N(t)}H(0)\|_{\dot{H}^\frac12}\Big)\lesssim N(t)^{1-s}.
\end{align*}
On the other hand,
\begin{align*}
|P(I_{N(t)}G(0))-P(G(0))\big|\lesssim&\big\|(Id-I_{N(t)})G(0)\big\|_{\dot{H}^\frac12}\|G(0)\|_{\dot{H}^\frac12}\\
\lesssim&\frac1{\sqrt{\lambda(0)}}\Big(\frac1{N(t)\lambda(0)}\Big)^\frac12\\
\lesssim&\frac1{\sqrt{\lambda(t)}}\Big(\frac1{N(t)\lambda(t)}\Big)^\frac12\\
\lesssim&\frac1{\lambda(t)^{1-\frac{1-\beta}{2\beta}}}.
\end{align*}

\end{proof}

\begin{proposition}[Almost conservation laws]\label{prop:almostcons}  Let $\frac{1}{1+\min\{1,\frac4d\}}<s<1,~\frac{4s}{\min\{4,d\}s-4(1-s)}<\frac1{\beta}$.
There holds the following control of the modified energy and momentum on $[0,T^+]$:
\begin{align}\label{equ:modiener}
\Big|E(I_{N(t)}u(t))+\frac{\Xi(t)}{\lambda(t)^2}\Big|\leq&\frac1{\lambda(t)^{2(1-\alpha_1)}},\\\label{equ:modimom}
\big|P(I_{N(t)}u(t))\big|\leq&\frac1{\lambda(t)^{1-\alpha_1}},
\end{align}
for some $\alpha_1=\frac{1-\beta}{4\beta}.$ In other words,
\begin{align}\label{equ:lamenermod}
\big|\lambda(t)^2E(I_{N(t)}u(t))+\Xi(t)|\leq&\lambda(t)^{2\alpha_1}\leq\Gamma_{b(t)}^{10},\\\label{equ:lammonmod}
\lambda(t)\big|P(I_{N(t)}u(t))\big|\leq&\lambda(t)^{\alpha_1}\leq\Gamma_{b(t)}^{10}.
\end{align}

\end{proposition}

\begin{proof}
Without loss of generality, we may assume $t=T^+$. By \eqref{equ:modenerlwp}, \eqref{equ:jk}, \eqref{equ:eint0}, \eqref{equ:dius},  we have
\begin{align*}
&\Big|E(I_{N(T^+)}u(T^+))+\frac{\Xi(T^+)}{\lambda(T^+)^2}\Big|\\\leq&\Big|E(I_{N(T^+)}u(0))+\frac{\Xi(T^+)}{\lambda(T^+)^2}\Big|
+\sum_{k=k_0}^{k^+}\sum_{j=1}^{J_k}\big|E(I_{N(T^+)}u(\tau_{k}^{j+1}))-E(I_{N(T^+)}u(\tau_{k}^{j}))\big|\\
\lesssim&N(T^+)^{2(1-s)}+\frac1{\lambda(T^+)^{2-\frac{1-\beta}{\beta}\frac2{d+2}}}\\
&+\sum_{k=k_0}^{k^+}kN(T^+)^\frac{2(1-s)}sN(T^+)^{-\min\{1,\frac4d\}s_+}\Big(\|I_{N(T^+)}u(\tau_k^j)\|_{H^1}^{2+\frac4d}
+\|I_{N(T^+)}u(\tau_k^j)\|_{H^1}^{2+\frac8d}\Big)\\
\lesssim&\Big(\frac1{\lambda(T^+)}\Big)^\frac{2(1-s)}{\beta}+\frac1{\lambda(T^+)^{2-\frac{1-\beta}{\beta}\frac2{d+2}}}\\
&+\sum_{k=k_0}^{k^+}kN(T^+)^\frac{2(1-s)}sN(T^+)^{-\min\{1,\frac4d\}s_+}
\Big[\Big(\frac{N(T^+)}{N(t_k)}\Big)^{1-s}\frac1{\lambda(t_k)}\Big]^{2+\frac8d}.
\end{align*}
We now sum up the geometric series from \eqref{equ:ltk} to get
\begin{align*}
&\sum_{k=k_0}^{k^+}kN(T^+)^\frac{2(1-s)}sN(T^+)^{-\min\{1,\frac4d\}s_+}
\Big[\Big(\frac{N(T^+)}{N(t_k)}\Big)^{1-s}\frac1{\lambda(t_k)}\Big]^{2+\frac8d}\\
\lesssim&N(T^+)^{(1-s)(2+\frac8d+\frac{2}s)-\min\{1,\frac4d\}s_+}\sum_{k=k_0}^{k^+}kN(t_k)^{(2+\frac8d)(\beta-(1-s))}\\
\lesssim&k^+N(T^+)^{(2+\frac8d)\beta+\frac{2(1-s)}s-\min\{1,\frac4d\}s_+}\\
\lesssim&|\log(\lambda(T^+))|\Big(\frac1{\lambda(T^+)}\Big)^{\frac1{\beta}((2+\frac8d)\beta+\frac{2(1-s)}s-\min\{1,\frac4d\}s_+)}\\
\lesssim&\Big(\frac1{\lambda(T^+)}\Big)^{\frac1{\beta}((2+\frac8d)\beta+\frac{2(1-s)}s-\min\{1,\frac4d\}s_+)}.
\end{align*}

Similarly,  using \eqref{equ:modenerlwp123}, \eqref{equ:jk}, \eqref{equ:pu0} and \eqref{equ:dius},  we estimate
\begin{align*}
&\big|P(I_{N(T^+)}u(T^+))\big|\\\leq&\big|P(I_{N(T^+)}u(0))\big|+\sum_{k=k_0}^{k^+}\sum_{j=1}^{J_k}
\big|P(I_{N(T^+)}u(\tau_{k}^{j+1}))-P(I_{N(T^+)}u(\tau_{k}^{j}))\big|\\
\lesssim&N(T^+)^{1-s}+\frac1{\lambda(T^+)^{1-\frac{1-\beta}{2\beta}}}
+\sum_{k=k_0}^{k^+}kN(T^+)^\frac{2(1-s)}sN(T^+)^{-\min\{1,\frac4d\}s_+}\|I_{N(T^+)}u(t_k)\|_{H^1}^{1+\frac4d}\\
\lesssim&\Big(\frac1{\lambda(T^+)}\Big)^\frac{1-s}\beta+\frac1{\lambda(T^+)^{1-\frac{1-\beta}{2\beta}}}
+\sum_{k=k_0}^{k^+}kN(T^+)^{\frac{2(1-s)}s-\min\{1,\frac4d\}s_+}\Big[\Big(\frac{N(T^+)}{N(t_k)}\Big)^{1-s}\frac1{\lambda(t_k)}\Big]^{2+\frac 4d-\frac1s}\\
\lesssim&\Big(\frac1{\lambda(T^+)}\Big)^\frac{1-s}\beta+\frac1{\lambda(T^+)^{1-\frac{1-\beta}{2\beta}}}
+\Big(\frac1{\lambda(T^+)}\Big)^{\frac1{\beta}((2+\frac8d-\frac1s)\beta+\frac{2(1-s)}s-\min\{1,\frac4d\}s_+)}.
\end{align*}

\end{proof}

%%%%%%%%%%%%%%%%%%%%%%%%%%%%%%%%%%%%%%%%%%%%%%%%%%

%%%%%%%%%%%%%%%%%%%%%%%%%%%%%%%%%%%%%%%%%%%%%%%%%%

\section{Proof of Proposition \ref{prop:main}}

In this section, we will show Proposition \ref{prop:main}, and then we conclude the proof of our main Theorem \ref{thm:main}.

\subsection{Control of the geometrical parameters}

Recall the geometrical decoposition
\begin{equation}\label{equ:decass123}
u(t,x)=\frac1{\lambda(t)^\frac{d}2}\big(Q_{b(t)}(\cdot)+\var(t,\cdot)\big)\Big(\frac{x-x(t)}{\lambda(t)}\Big)e^{-i\gamma(t)},~t\in[0,T^+].
\end{equation}
Let us introduce the rescaled time
\begin{equation}\label{equ:stdef}
\frac{ds}{dt}=\frac{1}{\lambda(s)^2}\quad\text{with}\quad s(0)=s_0=e^\frac{5\pi}{9b(0)}
\end{equation}
and  $y=\frac{x-x(t)}{\lambda(t)}$. Then,  $\var(s,y)$ satisfies on $[0,T^+]$ the equation:
\begin{align*}
&i\frac{\pa Q_b}{\pa b}b_s+i\pa_s\var+\Delta Q_b+\Delta\var+|Q_b+\var|^\frac4d(Q_b+\var)\\
=&-\Big(\gamma_s+i\frac{d}2\frac{\lambda_s}{\lambda}\Big)(Q_b+\var)+i\Big(\frac{x_s}{\lambda}+\frac{\lambda_s}{\lambda}y\Big)\cdot\nabla(Q_b+\var).
\end{align*}
To simplify notations, we note
$$\var=\var_1+i\var_2,~Q_b=\Sigma+i\Theta$$
in terms of real and imaginary parts. We have  by using Remark \ref{Rem:qbequ}
\begin{align}\nonumber
b_s\frac{\pa\Sigma}{\pa b}+\pa_s\var_1-M_-(\var)+b\Lambda\var_1=&\Big(\frac{\lambda_s}{\lambda}+b\Big)\Lambda\Sigma+\tilde{\gamma}_s\Theta+\frac{x_s}{\lambda}\cdot\nabla\Sigma\\\nonumber
&+\Big(\frac{\lambda_s}{\lambda}+b\Big)\Lambda\var_1+\tilde{\gamma}_s\var_2+\frac{x_s}{\lambda}\cdot\nabla\var_1\\\label{equ:var1}
&+{\rm Im}(\Psi_b)-R_2(\var)\\\nonumber
b_s\frac{\pa\Theta}{\pa b}+\pa_s\var_2+M_+(\var)+b\Lambda\var_2=&\Big(\frac{\lambda_s}{\lambda}+b\Big)\Lambda\Theta-\tilde{\gamma}_s\Sigma+\frac{x_s}{\lambda}\cdot\nabla\Theta\\\nonumber
&+\Big(\frac{\lambda_s}{\lambda}+b\Big)\Lambda\var_2-\tilde{\gamma}_s\var_1+\frac{x_s}{\lambda}\cdot\nabla\var_2\\\label{equ:var2}
&-{\rm Re}(\Psi_b)+R_1(\var),
\end{align}
with $\tilde{\gamma}(s)=-s-\gamma(s).$
The linear operator close to $Q_b$ is now a deformation of the linear operator $L$ close to $Q$ and is $M=(M_+,M_-)$ with
\begin{align*}
M_+(\var)=&-\Delta\var_1+\var_1-\Big(\frac{4\Sigma^2}{d|Q_b|^2}+1\Big)|Q_b|^\frac4d\var_1-\Big(\frac{4\Sigma\Theta}{d|Q_b|^2}|Q_b|^\frac4d\Big)\var_2\\
M_-(\var)=&-\Delta\var_2+\var_2-\Big(\frac{4\Theta^2}{d|Q_b|^2}+1\Big)|Q_b|^\frac4d\var_2-\Big(\frac{4\Sigma\Theta}{d|Q_b|^2}|Q_b|^\frac4d\Big)\var_1.
\end{align*}
The formally quadratic in $\var$ interaction terms are:
\begin{align*}
R_1(\var)=&(\var_1+\Sigma)|\var+Q_b|^\frac4d-\Sigma|Q_b|^\frac4d-\Big(\frac{4\Sigma^2}{d|Q_b|^2}+1\Big)|Q_b|^\frac4d\var_1-\Big(\frac{4\Sigma\Theta}{d|Q_b|^2}|Q_b|^\frac4d\Big)\var_2,\\
R_2(\var)=&(\var_2+\Theta)|\var+Q_b|^\frac4d-\Theta|Q_b|^\frac4d-\Big(\frac{4\Theta ^2}{d|Q_b|^2}+1\Big)|Q_b|^\frac4d\var_2-\Big(\frac{4\Sigma\Theta}{d|Q_b|^2}|Q_b|^\frac4d\Big)\var_1.
\end{align*}

The formally cubic terms in $\epsilon$ are:
\begin{equation*}
\begin{split}
\widetilde{R}_1(\epsilon)=&R_1(\epsilon)-\epsilon_1^2\frac{|Q_b|^\frac{4}{d}}{|Q_b|^4}\left[\frac{2}{d}\left(1+\frac{4}{d}\right)\Sigma^3+\frac{6}{d}\Sigma\Theta^2\right]\\
-&\epsilon_2^2\frac{|Q_b|^{\frac{4}{d}}}{|Q_b|^4}\left[\frac{2}{d}\Sigma^3+\frac{2}{d}\left(
\frac{4}{d}-1\right)\Sigma\Theta^2\right]\\
-&\frac{4}{d}\frac{|Q_b|^{\frac{4}{d}}}{|Q_b|^4}\epsilon_1\epsilon_2\left[\left(\frac{4}{d}-1\right)\Sigma^2\Theta+\Theta^3\right]
\end{split}
\end{equation*}

\begin{equation*}
\begin{split}
\widetilde{R}_2(\epsilon)=&R_2(\epsilon)-\epsilon_2^2\frac{|Q_b|^\frac{4}{d}}{|Q_b|^4}\left[\frac{2}{d}\left(1+\frac{4}{d}\right)\Theta^3+\frac{6}{d}\Theta\Sigma^2\right]\\
-&\epsilon_1^2\frac{|Q_b|^{\frac{4}{d}}}{|Q_b|^4}\left[\frac{2}{d}\Theta^3+\frac{2}{d}\left(
\frac{4}{d}-1\right)\Theta\Sigma^2\right]\\
-&\frac{4}{d}\frac{|Q_b|^{\frac{4}{d}}}{|Q_b|^4}\epsilon_1\epsilon_2\left[\left(\frac{4}{d}-1\right)\Theta^2\Sigma+\Sigma^3\right]
\end{split}
\end{equation*}

\begin{remark}
	From Weinstein\cite{Wen85}, the linearized operator $L$ close to the ground state in dimension $d$ can be explicitly written $L=(L_+,L_-)$ with
	$$L_+=-\Delta+1-\big(\tfrac4d+1\big)Q^\frac4d,\quad L_-=-\Delta+1-Q^\frac4d,$$
	and the following algebraic relations hold:
	\begin{align*}
	L_+(\Lambda Q)=-2Q,&~L_+(\nabla Q)=0,\\
	L_-(Q)=0,~L_-(yQ)=-2\nabla Q,~&L_-(|y|^2Q)=-4\Lambda Q.
	\end{align*}
	
\end{remark}

Recall
$$\big\|e^{Cr}(Q_b-Q)\big\|_{H^{10}\cap C^2}\to0\quad\text{as}\quad |b|\to0,$$
we can replace $Q_{b(s)}$ by $Q$ in the following with some loss such as $\Gamma_{b(s)}^{10}$.

\begin{lemma}[Estimates induced by conservation laws]\label{lem:estconlaw}
	For all $s\in [s_0,s^+]$ with $s^+=s(T^+)$, there holds:
	\begin{align}\nonumber
	&\Big|\big[2\big(\var_1,\Sigma+b\Lambda\Theta-{\rm Re}(\Psi_b)\big)+2\big(\var_2,\Theta-b\Lambda\Sigma-{\rm Im}(\Psi_b)\big)\big]-2\Xi(s)-\int\big|I_{N(s)\lambda(s)}\nabla\var(s)\big|^2\\\nonumber
	&+\Big[\int\Big(\frac{4}{d}+1\Big)|Q|^\frac4d(I_{N\lambda}\var_1)^2+\int|Q|^\frac4d(I_{N\lambda}\var_2)^2\Big]\Big|\\\label{equ:modenerg}
	%&-\Big[2\Xi(s)+\int\big|I_{N(s)\lambda(s)}\nabla\var(s)\big|^2\Big]\Big|\\
	\leq&\delta_0\Big(\int\big|I_{N(s)\lambda(s)}\nabla\var(s)\big|^2+\int|\var(s)|^2e^{-|y|}\Big)+\Gamma_{b(s)}^{1-C\eta},
	\end{align}
	and
	\begin{equation}
	\label{equ:modmomet}
	\big|(\var_2,\nabla Q)\big|\leq\delta_0\Big(\int\big|I_{N(s)\lambda(s)}\nabla\var(s)\big|^2\Big)^\frac12+\Gamma_{b(s)}^{10}.
	\end{equation}

\end{lemma}

\begin{proof}
	Note that
	$$I_Nu(t,x)=\frac1{\lambda^{\frac{d}{2}}}\big[I_{N\lambda}(Q_b+\var)\big]\Big(\frac{x-x(t)}{\lambda(t)}\Big),$$
	we have
	\begin{align}\nonumber
	2\lambda^2E(I_Nu)=&2E\big(I_{N\lambda}(Q_b+\var)\big)\\\nonumber
	=&2E\big(I_{N\lambda}(Q_b+\var)\big)-2E\big(Q_b+I_{N\lambda}\var\big)\\\label{equ:modenr1}
	&+2E\big(Q_b+I_{N\lambda}\var\big).
	\end{align}
	On the other hand, observe that
	$$\Delta Q_b-Q_b+ib\Lambda Q_b+|Q_b|^\frac4dQ_b=-\Psi_b,$$
	we get
	\begin{align*}
	&2E\big(Q_b+I_{N\lambda}\var\big)=\int\big|\nabla(Q_b+I_{N\lambda}\var)\big|^2-\frac{d}{d+2}\int|Q_b+I_{N\lambda}\var|^\frac{2(d+2)}{d}\\
	=&\int\Big(|\nabla Q_b|^2-2{\rm Re}\big(\Delta Q_b\cdot\overline{ I_{N\lambda}\var}\big)+|\nabla I_{N\lambda}\var|^2\Big)-\frac{d}{d+2}\int|Q_b+I_{N\lambda}\var|^\frac{2(d+2)}{d}\\
	=&2E(Q_b)+\int|\nabla I_{N\lambda}\var|^2-2{\rm Re}\int(\Delta Q_b+|Q_b|^\frac4dQ_b)\overline{I_{N\lambda}\var}\\
	&-\frac{d}{d+2}\int\Big(|Q_b+I_{N\lambda}\var|^\frac{2(d+2)}{d}-|Q_b|^\frac{2(d+2)}{d}-\frac{2(d+2)}{d}
	|Q_b|^\frac4d{\rm Re}\big(Q_b\overline{I_{N\lambda}\var}\big)\Big)
	\end{align*}
	and
	\begin{align*}
	&2{\rm Re}\int(\Delta Q_b+|Q_b|^\frac4dQ_b)\overline{I_{N\lambda}\var}=2{\rm Re}\int\big(Q_b-ib\Lambda Q_b-\Psi_b\big)\overline{I_{N\lambda}\var}\\
	=&2{\rm Re}\int\big(Q_b-ib\Lambda Q_b-\Psi_b\big)\overline{\var}
	+2{\rm Re}\int(I_{N\lambda}-{\rm Id})\big(Q_b-ib\Lambda Q_b-\Psi_b\big)\overline{\var}\\
	=&2\big(\var_1,\Sigma+b\Lambda\Theta-{\rm Re}(\Psi_b)\big)+2\big(\var_2,\Theta-b\Lambda\Sigma-{\rm Im}(\Psi_b)\big)\\
	&+2{\rm Re}\int(I_{N\lambda}-{\rm Id})\big(Q_b-ib\Lambda Q_b-\Psi_b\big)\overline{\var}.
	\end{align*}
	Hence,
	\begin{equation}\label{conservationofenergy}
	\begin{split}
	&2\big(\var_1,\Sigma+b\Lambda\Theta-{\rm Re}(\Psi_b)\big)+2\big(\var_2,\Theta-b\Lambda\Sigma-{\rm Im}(\Psi_b)\big)-2\Xi(s)-\int\big|I_{N(s)\lambda(s)}\nabla\var(s)\big|^2\\
	&+\int\Big(\frac{4\Sigma^2}{d|Q_b|^2}+1\Big)|Q_b|^\frac4d(I_{N\lambda}\var_1)^2+\int\Big(\frac{4\Theta^2}{d|Q_b|^2}+1\Big)|Q_b|^\frac4d(I_{N\lambda}\var_2)^2\\
	=&2E(Q_b)-2\big(\lambda^2E(I_Nu)+\Xi(s)\big)+2E\big(I_{N\lambda}(Q_b+\var)\big)-2E\big(Q_b+I_{N\lambda}\var\big)\\
	&-2{\rm Re}\int(I_{N\lambda}-{\rm Id})\big(Q_b-ib\Lambda Q_b-\Psi_b\big)\overline{\var}\\
	&-8\int\frac{\Sigma\Theta}{d|Q_b|^2}|Q_b|^\frac4dI_{N\lambda}\var_1I_{N\lambda}\var_2
	-\frac{d}{d+2}\int J(I_{N\lambda}\var),
	\end{split}
	\end{equation}
	where the cubic term $J(\var)$ ( is given in Appendix B in \cite{MR06JAMS})
	\begin{align*}
	J(\var)=&|\var+Q_b|^{2+\frac4d}-|Q_b|^{2+\frac4d}-\Big(2+\frac4d\Big)|Q_b|^{\frac4d}(\Sigma\var_1+\Theta\var_2)\\
	&-\var_1^2|Q_b|^{\frac{4}{d}-2}\Big[\Big(1+\frac2d\Big)\Big(1+\frac4d\Big)\Sigma^2+\Big(1+\frac2d\Big)\Theta^2\Big]\\
	&-\var_2^2|Q_b|^{\frac{4}{d}-2}\Big[\Big(1+\frac2d\Big)\Big(1+\frac4d\Big)\Theta^2+\Big(1+\frac2d\Big)\Sigma^2\Big]\\
	&-\var_1\var_2|Q_b|^{\frac4d-2}\frac8d\Big(1+\frac2d\Big)\Sigma\Theta.
	\end{align*}
	Indeed, $J(\epsilon)$ is the formal cubic term in $\epsilon$ in the expansion
	$$ |\epsilon+Q_b|^{2+\frac{4}{d}}=|Q_b|^{2+\frac{4}{d}}\left[1+\frac{2(\Sigma\epsilon_1+\Theta\epsilon_2)}{|Q_b|^2}+\frac{|\epsilon|^2}{|Q_b|^2}\right]^{\frac{d+2}{d}}.
	$$
	In fact, $J(\epsilon)$ is not cubic in $\epsilon$ when $d\geq 4$, but from elementary inequality, $J(\epsilon)=O(|\epsilon|^{2+\frac{4}{d}})$. This can give us the desired smallness in the sequel.
	
	Then, \eqref{equ:modenerg} follows by \eqref{equ:eqb}, Proposition \ref{prop:almostcons}   and
	\begin{equation}
	\big\|(I_{N(t)\lambda(t)}-{\rm Id})Q\big\|_{H^p}\lesssim \lambda(t)^{C_p}\lesssim\Gamma_{b(t)}^{10}.
	\end{equation}
	
	Next, we turn to prove \eqref{equ:modmomet}. Observe that
	\begin{align*}
	\lambda P(I_Nu)=&P\big(I_{N\lambda}(Q_b+\var)\big)\\
	=&\Big(P\big(I_{N\lambda}(Q_b+\var)\big)-P(Q_b+I_{N\lambda}\var)\Big)+P(Q_b+I_{N\lambda}\var)
	\end{align*}
	and
	\begin{align*}
	P(Q_b+I_{N\lambda}\var)=&{\rm Im}\int\nabla(Q_b+I_{N\lambda}\var)(\overline{Q_b+I_{N\lambda}\var})\\
	=&-2\int I_{N\lambda}\var_2\nabla\Sigma-2\int\Theta\nabla\Sigma+2\int I_{N\lambda}\var_1\nabla\Theta+2\int I_{N\lambda}\var_1\nabla I_{N\lambda}\var_2\\
	=&-2(\var_2,\nabla Q)+2(\var_2,\nabla(Q-\Sigma))+2(\var_2,({\rm Id}-I_{N\lambda})\nabla\Sigma)\\
	&-2\int\Theta\nabla\Sigma+2\int I_{N\lambda}\var_1\nabla\Theta+2\int I_{N\lambda}\var_1\nabla I_{N\lambda}\var_2,
	\end{align*}
	then, we obtain
	\begin{align*}
	-2(\var_2,\nabla Q)=&\lambda P(I_Nu)-\Big(P\big(I_{N\lambda}(Q_b+\var)\big)-P(Q_b+I_{N\lambda}\var)\Big)\\
	&-2(\var_2,\nabla(Q-\Sigma))-2(\var_2,({\rm Id}-I_{N\lambda})\nabla\Sigma)\\
	&+2\int\Theta\nabla\Sigma-2\int I_{N\lambda}\var_1\nabla\Theta-2\int I_{N\lambda}\var_1\nabla I_{N\lambda}\var_2.
	\end{align*}
	Note that $(\Theta,\nabla\Sigma)=0$ since $Q_b$ is radial. Thus, we get \eqref{equ:modmomet} by \eqref{equ:qbclsq} and Proposition \ref{prop:almostcons}. Therefore, we conclude the proof of Lemma \ref{lem:estconlaw}.

\end{proof}

Recall
\begin{equation}\label{equ:qb0paqb}
(Q_b)|_{b=0}=Q,\quad \Big(\frac{\pa Q_b}{\pa b}\Big)|_{b=0}=-i\frac{|y|^2}{4}Q.
\end{equation}
It is easy to see that
\begin{equation}\label{equ:commut}
\nabla\Lambda f-\Lambda\nabla f=\nabla f,~(f,\Lambda g)=-(\Lambda f,g).
\end{equation}

\begin{lemma}
There holds

$(i)$
\begin{align}
\label{equ:bsest}
&b_s\Big[\big(\tfrac{\pa\Theta}{\pa b},\Lambda\Sigma\big)-\big(\tfrac{\pa\Sigma}{\pa b},\Lambda\Theta\big)+\big(\var_1,\tfrac{\pa\Lambda\Theta}{\pa b}\big)-\big(\var_2,\tfrac{\pa\Lambda\Sigma}{\pa b}\big)\Big]\\\nonumber
=&-(M_+(\var),\Lambda\Sigma)-(M_-(\var),\Lambda\Theta)-\tilde{\gamma}_s\big[(\var_1,\Lambda\Sigma)+(\var_2,\Lambda\Theta)\big]\\\nonumber
&-\frac{x_s}{\lambda}\cdot\Big[(\nabla\Sigma,\Lambda\Theta)-(\nabla\Theta,\Lambda\Sigma)+(\var_2,\nabla\Lambda\Sigma)-(\var_1,\nabla\Lambda\Theta)\Big]\\\nonumber
&+(R_1(\var),\Lambda\Sigma)+(R_2(\var),\Lambda\Theta)-({\rm Re}(\Psi_b),\Lambda\Sigma)-({\rm Im}(\Psi_b),\Lambda\Theta),
\end{align}

$(ii)$
\begin{align}
\label{equ:lamsest}
&\Big(\frac{\lambda_s}{\lambda}+b\Big)\big[(\Lambda\Sigma,|y|^2\Sigma)+(\Lambda\Theta,|y|^2\Theta)\big]\\\nonumber
=&(M_+(\var),|y|^2\Theta)-(M_-(\var),|y|^2\Sigma)-\frac{\lambda_s}{\lambda}\big[(\var_1,\Lambda(|y|^2\Sigma))+(\var_2,\Lambda(|y|^2\Theta))\big]\\\nonumber
&+b_s\Big[\big(\tfrac{\pa\Sigma}{\pa b},|y|^2\Sigma\big)+\big(\tfrac{\pa\Theta}{\pa b},|y|^2\Theta\big)-\big(\var_1,\tfrac{\pa(|y|^2\Sigma)}{\pa b}\big)-\big(\var_2,\tfrac{\pa(|y|^2\Theta)}{\pa b}\big)\Big]\\\nonumber
&+\tilde{\gamma}_s\big[(\var_1,|y|^2\Theta)-(\var_2,|y|^2\Sigma)\big]+\frac{x_s}{\lambda}\cdot\big[(\var_1,|y|^2\nabla\Sigma)+(\var_2,|y|^2\nabla\Theta)\big]\\\nonumber
&-(R_1(\var),|y|^2\Theta)+(R_2(\var),|y|^2\Sigma)-({\rm Im}(\Psi_b),|y|^2\Sigma)+({\rm Re}(\Psi_b),|y|^2\Theta),
\end{align}

$(iii)$
\begin{align}
\label{equ:xsest}
&\frac{x_s}{\lambda}\cdot\big[(\nabla\Sigma,y\Sigma)+(\nabla\Theta,y\Theta)-(\var_1,\nabla(y\Sigma))-(\var_2,\nabla(y\Theta))\big]\\\nonumber
=&(M_+(\var),y\Theta)-(M_-(\var),y\Sigma)+\tilde{\gamma}_s\big[(\var_1,y\Theta)-(\var_2,y\Sigma)\big]\\\nonumber
&-b_s\Big[\big(\var_1,y\tfrac{\pa\Sigma}{\pa b}\big)+\big(\var_2,y\tfrac{\pa\Theta}{\pa b}\big)\Big]
+\frac{\lambda_s}{\lambda}\big[(\var_1,\Lambda(y\Sigma))+(\var_2,\Lambda(y\Theta))]\\\nonumber
&-(R_1(\var),y\Theta)+(R_2(\var),y\Sigma)-({\rm Im}(\Psi_b),y\Sigma)+({\rm Re}(\Psi_b),y\Theta),
\end{align}

$(iv)$
\begin{align}\label{equ:gammasest}
&\tilde{\gamma}_s\big[(\Theta,\Lambda^2\Theta)+(\Sigma,\Lambda^2\Sigma)+(\var_1,\Lambda^2\Sigma)+(\var_2,\Lambda^2\Theta)\big]\\\nonumber
=&(M_+(\var),\Lambda^2\Sigma)+(M_-(\var),\Lambda^2\Theta)+\tilde{\gamma}_s\big[(\var_1,\Lambda^2\Sigma)+(\var_2,\Lambda^2\Theta)\big]\\\nonumber
&-b_s\Big[\big(\tfrac{\pa\Sigma}{\pa b},\Lambda^2\Theta\big)-\big(\tfrac{\pa\Theta}{\pa b},\Lambda^2\Sigma\big)-\big(\var_1,\tfrac{\pa\Lambda^2\Theta}{\pa b}\big)+\big(\var_2,\tfrac{\pa\Lambda^2\Sigma}{\pa b}\big)\Big]\\\nonumber
&+\frac{x_s}{\lambda}\cdot\Big[(\nabla\Sigma,\Lambda^2\Theta)-(\nabla\Theta,\Lambda^2\Sigma)-(\var_1,\nabla\Lambda^2\Theta)+(\var_2,\nabla\Lambda^2\Sigma)\Big]\\\nonumber
&-\frac{\lambda_s}{\lambda}\big[(\var_1,\Lambda^3\Theta)-(\var_2,\Lambda^3\Sigma)\big]+2\Big(\frac{\lambda_s}{\lambda}+b\Big)(\Lambda\Sigma,\Lambda^2\Theta)\\\nonumber
&-(R_2(\var),\Lambda^2\Theta)-(R_1(\var),\Lambda^2\Sigma)+({\rm Im}(\Psi_b),\Lambda^2\Theta)+({\rm Re}(\Psi_b),\Lambda^2\Sigma).
\end{align}

\end{lemma}

\begin{proof}
Take the inner product of \eqref{equ:var1} with $-\Lambda\Theta$ and of \eqref{equ:var2} with $\Lambda\Sigma$, sum the obtained equalities,  use the orthogonality condition \eqref{equ:cth} and integrate by parts to get \eqref{equ:bsest}.

\eqref{equ:lamsest} follows by summing the inner product of \eqref{equ:var1} with $|y|^2\Sigma$ and of \eqref{equ:var2} with $|y|^2\Theta$ and  the orthogonality condition \eqref{equ:cy2}.

Similarly, \eqref{equ:xsest}  follows by summing the inner product of \eqref{equ:var1} with $y\Sigma$ and of \eqref{equ:var2} with $y\Theta$ and  the orthogonality condition \eqref{equ:cy1}.

Finally,
 \eqref{equ:gammasest}  follows by summing the inner product of \eqref{equ:var1} with $\Lambda^2\Theta$ and of \eqref{equ:var2} with $-\Lambda^2\Sigma$ and  the orthogonality condition \eqref{equ:cth2}.

\end{proof}

Note that
\begin{align*}
\Lambda(fg)=&g\Lambda f+f\Lambda g-\frac{d}{2}fg\\
=&g\Lambda f+f y\cdot\nabla g\\
\Delta(\Lambda f)=&2\Delta f+\Lambda(\Delta f),
\end{align*}
and using the  orthogonality condition \eqref{equ:cth2}, we obtain the following lemma.

\begin{lemma}
There holds
\begin{align}\label{equ:m+m-lamthe}
&-(M_+(\var),\Lambda\Sigma)-(M_-(\var),\Lambda\Theta)\\\nonumber
=&2\big(\var_1,\Sigma+b\Lambda\Theta-{\rm Re}(\Psi_b)\big)+2\big(\var_2,\Theta-b\Lambda\Theta-{\rm Im}(\Psi_b)\big)\\\nonumber
&-\big(\var_1,{\rm Re}(\Lambda\Psi_b)\big)-\big(\var_2,{\rm Im}(\Lambda\Psi_b)\big).
\end{align}
Thus, this together with \eqref{equ:bsest} and  \eqref{equ:z2est} yields that
\begin{align}\label{equ:bestfurth}
&b_s\left[(\partial_b\Theta,\Lambda\Sigma)-(\partial_b\Sigma,\Lambda\Theta)+(\epsilon_1,\partial_b\Lambda\Theta)-(\epsilon_2,\partial_b\Lambda\Sigma)\right]\\=&2\Xi(s)+H(I_{N\lambda}\var,I_{N\lambda}\var)-(\var_1,{\rm Re}(\Lambda\Psi_b))-(\var_2,{\rm Im}(\Lambda\Psi_b))\\\nonumber
&-\tilde{\gamma}_s\big[(\var_1,\Lambda\Sigma)+(\var_2,\Lambda\Theta)\big]-\frac{x_s}{\lambda}\cdot\Big[(\var_2,\nabla\Lambda\Sigma)-(\var_1,\nabla\Lambda\Theta)\Big]\\\nonumber
&+O\Big(\delta_0\Big(\int\big|I_{N(s)\lambda(s)}\nabla\var(s)\big|^2+\int|\var(s)|^2e^{-|y|}\Big)\Big)+O(\Gamma_{b(s)}^{1-C\eta})+F(s),
\end{align}
with
\begin{equation*}
\begin{split}
F(s)=&(R_2(\epsilon)-R_2(I_{N\lambda}\epsilon),\Lambda\Theta)+(R_1(\epsilon)-R_1(I_{N\lambda}\epsilon),\Lambda\Sigma)\\
+&\widetilde{H}_b(I_{N\lambda}\epsilon,I_{N\lambda}\epsilon)+(\widetilde{R}_1(I_{N\lambda}\epsilon),\Lambda\Sigma)+(\widetilde{R}_2(I_{N\lambda}\epsilon),\Lambda\Theta),
\end{split}
\end{equation*}
\begin{equation*}
\begin{split}
\widetilde{H}_b(\epsilon,\epsilon):=&\int |V_1(y)|(\epsilon_1)^2+\int V_2(y)(\epsilon_2)^2+\int V_{12}(y)\epsilon_1\epsilon_2,\\
V_1(y)=&\frac{2}{d}\left(1+\frac{4}{d}\right)\left[\frac{|Q_b|^{\frac{4}{d}}}{|Q_b|^4}\Sigma^3y\cdot\nabla\Sigma-Q^{\frac{4}{d}-1}y\nabla Q\right]\\
+&\frac{|Q_b|^{\frac{4}{d}}}{|Q_b|^4}\Theta^2\left[\frac{6}{d}\Sigma\Lambda\Sigma-\left(\frac{4}{d}+2\right)\Sigma^2-\Theta^2\right]\\
+&\frac{2}{d}\frac{|Q_b|^{\frac{4}{d}}}{|Q_b|^4}\Theta\Lambda\Theta\left[\Theta^2+\left(\frac{4}{d}-1\right)\Sigma^2\right],\\
V_2(y)=&\frac{2}{d}\left[\frac{|Q_b|^{\frac{4}{d}}}{|Q_b|^4}\Sigma^3y\cdot\nabla\Sigma-Q^{\frac{4}{d}-1}y\cdot\nabla Q\right]\\
+&\frac{|Q_b|^{\frac{4}{d}}}{|Q_b|^4}\Theta^2\left[\frac{2}{d}\left(
\frac{4}{d}-1\right)\Sigma\Lambda\Sigma-\left(
\frac{4}{d}+2\right)\Sigma^2-\left(
\frac{4}{d}+1\right)\Theta^2\right]\\
+&\frac{2\Theta\Lambda\Theta}{d}\frac{|Q_b|^{\frac{4}{d}}}{|Q_b|^4}\left[3\Sigma^2+\left(\frac{4}{d}+1\right)\Theta^2\right],\\
V_{12}(y)=&\frac{4}{d}\frac{|Q_b|^{\frac{4}{d}}}{|Q_b|^4}\left[\Theta\left(\Theta^2\Lambda\Sigma+\left(\frac{4}{d}-1\right)(\Sigma^2\Lambda\Sigma+\Sigma\Theta\Lambda\Theta)-2\Sigma|Q_b|^2\right)+\Sigma^3\Lambda\Theta\right].
\end{split}
\end{equation*}
Moreover, the remainder $F$ can be bounded by
$$|F|\leq \delta(\alpha^*)\left(\int|\nabla I_{N\lambda}\epsilon|^2+\int |\epsilon|^2e^{-|y|}\right)+\Gamma_b^{1-C\eta}.
$$
\end{lemma}
\begin{remark}
By the estimation of $F(s)$, we can absorb this term into $\Gamma_{b(s)}^{1-C\eta}.$  This is different from the dimension two case in \cite{CR09}.
\end{remark}
\begin{proof}
	The algebraic identity can be obtained directly from the proof of Lemma \ref{lem:estconlaw}. It is only a matter to estimate the remainder term $F(s)$. Denote by
	$$ F_1(s)=(R_2(\epsilon)-R_2(I_{N\lambda}\epsilon),\Lambda\Theta)+(R_1(\epsilon)-R_1(I_{N\lambda}\epsilon),\Lambda\Sigma),
	$$
	$$F_2(s)=F(s)-F_1(s).
	$$
Denote by
$$ \|I_{N\lambda}\epsilon\|_{H_{exp}^1}^2:=\int |I_{N\lambda}\epsilon|^2+\int |I_{N\lambda}\epsilon|^2e^{-|y|}.
$$
\underline{Estimate of $F_2$:} From \cite{MR03GAFA},
$$ |\widetilde{H}_b(I_{N\lambda}\epsilon,I_{N\lambda}\epsilon)|+|(\widetilde{R}_1(I_{N\lambda}\epsilon),\Lambda\Sigma)|+|(\widetilde{R}_2(I_{N\lambda}\epsilon),\Lambda\Theta)|\leq \delta_0\|I_{N\lambda}\epsilon\|_{H_{exp}^1}^2.
$$
\underline{Estimate of $F_1$:}
Note that $F_1$ can be written in the following form
\begin{equation*}
\begin{split}
F_1(s)=&\int |\epsilon|^2(\mathrm{Id}-I_{N\lambda})\phi_1+\int\left[|I_{N\lambda}\epsilon|^2-I_{N\lambda}(|\epsilon|^2)\right]\phi_2\\
+&O\left(\int \left(|\epsilon|^3\chi_{d=3}+|\epsilon|^{2+\frac{2}{d}}\chi_{d\geq 4}\right)\phi_3\right)
\end{split}
\end{equation*}
where $\phi_1,\phi_2,\phi_3$ are Schwartz functions built on $Q$ which decay exponentially as $r\rightarrow\infty.$
We estimate
$$ \left|\int |\epsilon|^2(\mathrm{Id}-I_{N\lambda})\phi_1\right|\leq \|\epsilon\|_{L^2}\|(\mathrm{Id}-I_{N\lambda}\phi)\|_{L^{\infty}}\lesssim \frac{1}{N\lambda}\lesssim \Gamma_b^{10},
$$
and
\begin{equation*}
\begin{split}
&\left|\int \left[|I_{N\lambda}\epsilon|^2-I_{N\lambda}(|\epsilon|^2)\right]\phi_2\right|\\ \lesssim &\|(I_{N\lambda}\epsilon-\epsilon)|\phi_2|^{1/2}\|_{L^2}\|(I_{N\lambda}\epsilon+\epsilon)|\phi_2|^{1/2}\|_{L^2}
+\|\epsilon|\phi|^{1/2}\|_{L^2}^2\|(\mathrm{Id}-I_{N\lambda})\phi_2\|_{L^{\infty}}\\
\lesssim &\frac{1}{N\lambda}\lesssim \Gamma_b^{10}.
\end{split}
\end{equation*}
To estimate the third term in $F_1$, we only deal with the case $d\geq 4$ here. Denote by $J_{N\lambda}=\mathrm{Id}-I_{N\lambda}, $and from the assumption $s>\frac{1}{1+\frac{4}{d}}=\frac{d}{d+4}$, we know that $\frac{2d}{d-2s}>2+\frac{2}{d}$, we use Sobolev embedding to estimate
\begin{equation*}
\begin{split}
\left|\int |\epsilon|^{2+\frac{2}{d}}\phi_3\right|\leq & C\int|I_{N\lambda}\epsilon|^{2+\frac{2}{d}}|\phi_3|+C\int |J_{N\lambda}\epsilon|^{2+\frac{2}{d}}|\phi_3|,\\
\|I_{N\lambda}\epsilon\|_{L^{2+\frac{2}{d}}(|\phi_3|dy)}\leq &\|I_{N\lambda}\epsilon\|_{L^2(|\phi_3|dy)}^{\theta_1}\|I_{N\lambda}\epsilon\|_{L^{\frac{2d}{d-2}}(|\phi_3|dy)}^{1-\theta_1}\\
\leq &C\|I_{N\lambda}\epsilon\|_{L^2(|\phi_3|dy)}^{\theta_1}\|\nabla I_{N\lambda}\epsilon\|_{L^2}^{1-\theta_1}
\leq C\|I_{N\lambda}\epsilon\|_{H_{exp}^1},
\end{split}
\end{equation*}
and
\begin{equation*}
\begin{split}
\|J_{N\lambda}\epsilon\|_{L^{2+\frac{2}{d}}(|\phi_3|dy)}\leq &\|J_{N\lambda}\epsilon\|_{L^2(|\phi_3|dy)}^{\theta_2}\|J_{N\lambda}\epsilon\|_{L^{\frac{2d}{d-2s}}(|\phi_3|dy)}^{1-\theta_2}\\
\leq & \|J_{N\lambda}\epsilon\|_{L^2}^{\theta_2}\|J_{N\lambda}\epsilon\|_{\dot{H}^s}^{1-\theta_2}
\\ \leq  &\frac{1}{(N\lambda)^{ s\theta_2}}\|J_{N\lambda}\epsilon\|_{\dot{H}^s}\\ \leq &\frac{1}{(N\lambda)^{s\theta_2}}\|P_{\geq N}\epsilon\|_{\dot{H}^s}\\ \leq &\frac{1}{(N\lambda)^{s\theta_2}}\|\nabla I_{N\lambda}\epsilon\|_{L^2}\leq C\|I_{N\lambda}\epsilon\|_{H_{exp}^1}.
\end{split}
\end{equation*}
This completes the proof of Lemma3.4.

\end{proof}

View the four equalities \eqref{equ:lamsest}, \eqref{equ:xsest}, \eqref{equ:gammasest}, \eqref{equ:bestfurth} as a linear system, which is invertible, we obtain
\begin{lemma}[Control of the geometrical parameters]\label{lem:congeo}
There holds:
\begin{align}\nonumber
\Big|\frac{\lambda_s}{\lambda}+b\Big|+|b_s|\leq&C\Big(\Xi(s)+\int|\nabla I_{N(s)\lambda(s)}\var(s)|^2+\int|\var(s)|^2e^{-|y|}\Big)\\\label{equ:lambsest}
&+\Gamma_{b(s)}^{1-C\eta},\\\nonumber
\Big|\tilde{\gamma}_s-\frac{(\var_1,L_+\Lambda^2Q)}{\|\Lambda Q\|_{L^2}^2}\Big|+\Big|\frac{x_s}{\lambda}\Big|\leq&\delta_0\Big(\int|\var(s)|^2e^{-|y|}\Big)^\frac12+\Gamma_{b(s)}^{1-C\eta}\\\label{equ:gammaxsest}
&+C\Big(\Xi(s)+\int|\nabla I_{N(s)\lambda(s)}\var(s)|^2\Big).
\end{align}

\end{lemma}

\subsection{Virial dispersion}

In this subsection, we will derive two global virial dispersive estimates at the heart of the control of the log-log regime in \cite{MR03GAFA,MR06JAMS}.  We begin with the global virial estimate first established in \cite{MR05annmath,MR03GAFA}.

\begin{lemma}[Global virial estimate, \cite{CR09}]\label{lem:glvirest}
We have
\begin{equation}\label{equ:glovirest}
b_s\geq c_0\Big(\Xi(s)+\int|\nabla I_{N(s)\lambda(s)}\var(s)|^2+\int|\var(s)|^2e^{-|y|}\Big)-\Gamma_{b(s)}^{1-C\eta}.
\end{equation}

\end{lemma}

Next, we consider another dispersive control of a slightly different kind exhibited in \cite{MR04Invemath,MR06JAMS}. The main idea is that the profile $Q_b+\zeta_b$ should be a better approximation of the solution. Let us introduce a cut off parameter
\begin{equation}\label{equ:atdef}
A(t)=e^\frac{2a}{b(t)}\quad\text{so that}\quad \Gamma_b^{-\frac{a}2}\leq A\leq\Gamma_b^{-\frac{3a}2}
\end{equation}
for some small parameter $0<a\ll1$ and
$$\tilde{\zeta}=\chi\big(\tfrac{r}{A}\big)\zeta_b$$
where $\chi(r)$ is a smooth cut-off function with
\begin{equation*}
\chi(r)=\begin{cases}1\quad\text{if}\quad 0\leq r\leq\frac32\\
0\quad\text{if}\quad r\geq2.
\end{cases}
\end{equation*}
We remark that $\tilde{\zeta}$ is a small Schwartz function due to the $A$ localization. We next consider the new variable
\begin{equation}\label{tildevardef}
\tilde{\var}=\var-\tilde{\zeta}.
\end{equation}
Then, by the same argument as  Lemma 6 in \cite{MR06JAMS} and the above, we obtain

\begin{lemma}[Virial dispersion in the radiative regime]\label{lem:virdradreg}
There holds for some universal constants $c_1>0$ and $s\in[s_0,s^+]$:
\begin{align}\label{equ:f1sdef}
\{f_1(s)\}_s\geq& c_1\Big(\Xi(s)+\int\big|\nabla I_{N(s)\lambda(s)}\tilde{\var}(s)\big|^2+\int|\widetilde{\var}(s)|^2e^{-|y|}+\Gamma_b\Big)\\\nonumber&-\frac1{\delta_1}\int_A^{2A}|\var|^2,
\end{align}
with
\begin{equation}\label{equ:f1defs}
f_1(s)=\frac{b}{4}\|yQ_b\|_{L^2}^2+\frac12{\rm Im}\Big(\int (y\cdot\nabla\tilde{\zeta})\var{\tilde{\zeta}}\Big)
+{\rm Re}(\var_2,\Lambda\tilde{\zeta})-{\rm Im}(\var_1,\Lambda\tilde{\zeta}).
\end{equation}
\end{lemma}

\begin{proof}
	\textbf{Step 1}: Algebraic dispersive relation.\\
	We have(See \cite{MR06JAMS} for the calculation of principle terms)
	\begin{equation}\label{algebraicfinal1}
	\begin{split}
	\frac{df_1}{ds}=&H(I_{N\lambda}\epsilon-\wit{\zeta}_b,I_{N\lambda}\epsilon-\wit{\zeta}_b)+(\epsilon_1-\Re\wit{\zeta}_b,\Re\Lambda F)+(\epsilon_2-\Im\wit{\zeta}_b,\Im\Lambda F)\\
	-&2(\lambda^2E(I_Nu)+\Xi(s))+b_s\left[(\epsilon_2-\Im\wit{\zeta},\partial_b\Lambda(\Sigma+\Re\wit{\zeta}))-(\epsilon_1-\Re\wit{\zeta},\partial_b\Lambda(\Theta+\Im\wit{\zeta}))\right]\\
	-&\frac{A_s}{A^2}\left[(\epsilon_2-\Im\wit{\zeta},\Lambda(y\cdot(\nabla\chi)\left(\frac{y}{A}\right)\Re\wit{\zeta}))-(\epsilon_1-\Re\wit{\zeta},\Lambda(y\cdot(\nabla\chi)\left(\frac{y}{A}\right)\Im\wit{\zeta}))\right]\\
	-&\left(\frac{\lambda_s}{\lambda}+b\right)\left[(\epsilon_2-\Im\wit{\zeta},\Lambda^2(\Sigma+\Re\wit{\zeta}))-(\epsilon_1-\Re\wit{\zeta},\Lambda^2(\Theta+\Im\wit{\zeta}))\right]\\
	-&\wit{\gamma}_s\left[(\epsilon_1-\Re\wit{\zeta},\Lambda(\Sigma+\Re\wit{\zeta}))+(\epsilon_2-\Im\wit{\zeta},\Lambda(\Theta+\Im\wit{\zeta}))\right]\\
	-&\frac{x_s}{\lambda}\cdot\left[(\epsilon_2-\Im\wit{\zeta},\nabla\Lambda(\Sigma+\Re\wit{\zeta}))-(\epsilon_1-\Re\wit{\zeta},\nabla\Lambda(\Theta+\Im\wit{\zeta}))\right]\\
	+&(R_1(\epsilon),\Lambda\Re\wit{\zeta})+(R_2(\epsilon),\Lambda\Im\wit{\zeta})\\
	+&(\epsilon_1-\Re\wit{\zeta}_b,\Lambda(\left(1+\frac{4}{d}\right)Q^{\frac{4}{d}}\Im\wit{\zeta_b}))+(\epsilon_2-\Im\wit{\zeta}_b,\Lambda(Q^{\frac{4}{d}}\Im\wit{\zeta}_b))\\
	+&(\epsilon_1,\wit{L})+(\epsilon_2,\wit{K})+\Upsilon_{N\lambda}+2\Xi(s)+\mathcal{R}emainder
	\end{split}
	\end{equation}
	with
	\begin{equation*}
	\begin{split}
	\wit{L}=&\left[\left(\frac{4\Sigma^2}{d|Q_b|^2}+1\right)|Q_b|^{\frac{4}{d}}-\left(1+\frac{4}{d}\right)Q^{\frac{4}{d}}\right]\Lambda\Re\wit{\zeta}_b+\frac{4\Sigma\Theta}{d|Q_b|^2}|Q_b|^{\frac{4}{d}}\Lambda\Im\wit{\zeta}_b,\\
	\wit{K}=&\left[\left(\frac{4\Theta^2}{d|Q_b|^2}+1\right)|Q_b|^{\frac{4}{d}}-Q^{\frac{4}{d}}\right]\Lambda\Im\wit{\zeta}_b+\frac{4\Sigma\Theta}{d|Q_b|^2}|Q_b|^{\frac{4}{d}}\Lambda\Re\wit{\zeta}_b,
	\end{split}
	\end{equation*}
	coming from the error term:
	\begin{equation*}
	\begin{split}
	&(L_+\epsilon_1+b\Lambda\epsilon_2,\Lambda\Re\wit{\zeta}_b)+(L_-\epsilon_2-b\Lambda\epsilon_1,\Lambda\Im\wit{\zeta}_b)\\
	-&(M_+(\epsilon)+b\Lambda\epsilon_2,\Lambda\Re\wit{\zeta}_b)-(M_-(\epsilon)-b\Lambda\epsilon_1,\Lambda\Im\wit{\zeta}_b).
	\end{split}
	\end{equation*}
	Another error terms
	\begin{equation*}
	\begin{split}
	\Upsilon_{N\lambda}=&((I_{N\lambda}-Id)\epsilon_1,\Re\Lambda\Psi_b)+((I_{N\lambda}-Id)\epsilon_2,\Im\Lambda\Psi_b)\\
	+&(L_+(I_{N\lambda}-Id)\epsilon_1+b\Lambda(I_{N\lambda}-Id)\epsilon_2,\Lambda\Re\wit{\zeta}_b)\\
	+&(L_-(I_{N\lambda}-Id)\epsilon_2-b\Lambda(I_{N\lambda}-Id)\epsilon_1,\Lambda\Im\wit{\zeta}_b)\\
	+&((I_{N\lambda}-Id)\epsilon_1,\Lambda(\Re F+\left(1+\frac{4}{d}\right)Q^{\frac{4}{d}}\Re\wit{\zeta}_b))\\
	+&((I_{N\lambda}-Id)\epsilon_2,\Lambda(\Im F+Q^{\frac{4}{d}}\Im\wit{\zeta}_b))
	\end{split}
	\end{equation*}
	and
	\begin{equation*}
	\begin{split}
	\mathcal{R}eminder=&\widetilde{H}_b(I_{N\lambda}\epsilon,I_{N\lambda}\epsilon)+(R_1(\epsilon)-R_1(I_{N\lambda}\epsilon),\Lambda\Sigma)+(R_2(\epsilon)-R_2(I_{N\lambda}\epsilon),\Lambda\Theta)\\
	+&(\wit{R}_1(I_{N\lambda}\epsilon),\Lambda\Sigma)+(\wit{R}_2(I_{N\lambda}\epsilon),\Lambda\Theta)-\frac{d}{d+2}\int J(I_{N\lambda}\epsilon)\\
	+&2E(I_{N\lambda}(Q_b+\epsilon))-2E(Q_b+I_{N\lambda}\epsilon)-2\Re\int(I_{N\lambda}-\mathrm{Id})(Q_b-ib\Lambda Q_b-\Psi_b)\ov{\epsilon}.
	\end{split}
	\end{equation*}
	
	%%%%%%%%%%%%%%%%%%%%%%%%%%%%%%%%%%%%%%%%
	\textbf{Step 2}: Control of interaction terms.\\
	
	\underline{Claim 1:}\begin{equation*}
	\begin{split}
	&\int |(\mathrm{Id}-I_{N\lambda})\epsilon|^2e^{-|y|}=O(\Gamma_b^{1+z_0}),\\
	&\int |\nabla I_{N\lambda}\wit{\epsilon}|^2+\int|\wit{\epsilon}|^2e^{-|y|}=\int|\nabla(I_{N\lambda}\epsilon-\wit{\zeta}_b)|^2+\int|I_{N\lambda}\epsilon-\wit{\zeta}_b|^2e^{-|y|}+O(\Gamma_b^{1+z_0}),\\
	&\int |\epsilon|^2e^{-|y|}\leq 2\int|\wit{\epsilon}|^2e^{-|y|}+\Gamma_b^{1+z_0}\\
	&\int |\nabla I_{N\lambda}\epsilon|^2\leq C\int |\nabla I_{N\lambda}\wit{\epsilon} |^2+\Gamma_b^{1-C\eta}
	\end{split}
	\end{equation*}
	Indeed, the first estimate comes from the boundness of $I_{N\lambda}\epsilon$ in $H^1$( thus the boundness of $\epsilon$ in $H^s$) as well as $(N\lambda)^{-s}=O(\Gamma_b^{1+z_0})$ by our bootstrap assumption. Similarly, the second estimate comes from the smallness of $\wit{\zeta}_b$:
	$$ \|\wit{\zeta}_b\|_{H^1}\leq \Gamma_b^{1-C\eta}
	$$
	and the error estimate $\|\mathrm{Id}-I_{N\lambda}\|_{H^1\rightarrow L^2}\leq \frac{1}{N\lambda}$. The third inequality comes from the property of $\zeta:$
	$$ \|\zeta_b(y)e^{-\frac{\sigma\theta(b|y|)}{b}}\|_{L^{\infty}(|y|\leq R_b)}\leq \Gamma_b^{\frac{1}{2}+\frac{\sigma}{10}},\forall \sigma\in(0,5),\quad \|\zeta_b |y|^{d/2}\|_{L^{\infty}(|y|>R_b)}\leq \Gamma_b^{\frac{1}{2}-C\eta}.
	$$
	\begin{equation*}
	\begin{split}
	\int |\epsilon|^2e^{-|y|}\leq &2\int |\wit{\epsilon}|^2e^{-|y|}+2\int |\wit{\zeta}|^2e^{-|y|}\\
	\leq &2\int |\wit{\epsilon}|^2e^{-|y|}+2\int \chi_A(y)^2|\zeta_b|^2e^{-\frac{\theta(b|y|)}{|b|}}\\
	\leq &2\int |\wit{\epsilon}|^2e^{-|y|}+2\Gamma_b^{1+\frac{1-C\eta}{5}}R_b+2\Gamma_b^{1-2C\eta}\log(2A)e^{-\frac{\theta(2\sqrt{1-\eta})}{|b|}}\\
	\leq &2\int |\wit{\epsilon}|^2e^{-|y|}+\Gamma_b^{1+z_0}.
	\end{split}
	\end{equation*}

	\underline{Claim 2:}\\
	\begin{enumerate}
		\item For $d\geq 3$ and any $B\geq 2$,
		$$ \int_{|y|\leq B}|\wit{\epsilon}^2|\leq CB^2\left(\int|\nabla I_{N\lambda}\wit{\epsilon}|^2+\int |\wit{\epsilon}|^2e^{-|y|}\right)+\Gamma_b^{1+z_0}
		$$
		\item Second order interaction : for $R(\epsilon)=R_1(\epsilon)$ or $R_2(\epsilon)$,
		$$ \int |R(\epsilon)|e^{-(1-C\eta)\frac{\theta(b|y|)}{|b|}}\leq  C\left(\int |\nabla I_{N\lambda}\wit{\epsilon}|^2+\int |\wit{\epsilon}|^2e^{-|y|}\right)+\Gamma_b^{1+z_0}.
		$$
		\item Small second-order interaction:
		$$ \int |R(\epsilon)|(|\wit{\zeta}_b|+|y\cdot\nabla\wit{\zeta}_b|)\leq \delta(\alpha^*)\left(\int |\nabla I_{N\lambda}\wit{\epsilon}|^2+\int |\wit{\epsilon}|^2e^{-|y|}\right)+\Gamma_b^{1+z_0}
		$$
		\item Small second-order scalar products: For any polynomial $P(y)$ and integers $0\leq k\leq 2, 0\leq l\leq 1$, there exists $C>0$ such that
		$$\int |\wit{\epsilon}||P(y)|\left(|\nabla_y^{k}\wit{\zeta}_b|+|\nabla_y^l\partial_b\wit{\zeta}_b|\right)\leq \Gamma_b^C\left(\int|\nabla I_{N\lambda}\wit{\epsilon}|^2+\int |\wit{\epsilon}|^2e^{-|y|}\right)^{\frac{1}{2}}+\Gamma_b^{\frac{1}{2}+z_0}
		$$
		\item Cut-off $\chi_A$ induced estimates
		$$\int |\epsilon|(|F|+|y\cdot F|)\leq C\Gamma_b^{\frac{1}{2}}\left(\int_{A\leq|y|\leq 2A}|\epsilon|^2\right)^{\frac{1}{2}}.
		$$
	\end{enumerate}
	\underline{Proof:}\\
	(1) follows from $\|\wit{\epsilon}\|_{L^2(|y|\leq B)}\leq \|I_{N\lambda}\wit{\epsilon}\|_{L^2(|y|\leq B)}+\|(\textrm{Id}-I_{N\lambda})\wit{\epsilon}\|_{L^2}$ and classical inequality(see \cite{MR06JAMS})
	$$ \int_{|y|\leq B}|v|^2\leq CB^2\left(\int |\nabla v|^2+\int |v|^2e^{-|y|}\right),
	$$
	combining with the first inequality in claim 1.
	
	For (2), from the classical inequality (see \cite{MR06JAMS})
	$$ |R(\epsilon)|\leq C|\epsilon|^2e^{-\left(\frac{4}{d}-1\right)\left(1-C\eta\right)\frac{\theta(b|y|)}{|b|}}+|\epsilon|^{1+\frac{4}{d}},\quad d\leq 3,
	$$
	$$ |R(\epsilon)|\leq C\min\left(|\epsilon|^2e^{-\left(\frac{4}{d}-1\right)\left(1-C\eta\right)\frac{\theta(b|y|)}{|b|}},|\epsilon|^{1+\frac{4}{d}}\right),\quad d\geq 4,
	$$
	Using the classical inequality
	$$ \int |v|^2e^{-\kappa |y|}\leq C\left(\int |\nabla v|^2+\int |v|^2e^{-|y|}\right)
	$$
	and the third inequality in claim 1, for $d\geq 4$, we have
	$$ \int |\epsilon|^2e^{-\left(\frac{4}{d}-1\right)(1-C\eta)\frac{\theta(b|y|)}{|b|}}\leq \int |\wit{\epsilon}|^2e^{-\frac{4}{d}(1-C\eta)\frac{\theta(b|y|)}{|b|}}+\Gamma_b^{1+z_0},
	$$
	we conclude by replacing $\wit{\epsilon}$ to $I_{N\lambda}\wit{\epsilon}$ and an error term can be absorbed into $\Gamma_b^{1+z_0}$. For $d=3, $ we easily have
	$$ \int |\epsilon|^{\frac{7}{7}}e^{-(1-C\eta)\frac{\theta(b|y|)}{|b|}}\leq C\int |\wit{\epsilon}|^{\frac{7}{3}}e^{-(1-C\eta)\frac{\theta(b|y|)}{|b|}}+\Gamma_b^{1+z_0}
	$$
	since $\frac{7}{3}>2$.
	
	For (4), we do the case $d\geq 4$here . We estimate it by
	\begin{equation*}
	\begin{split}
	&\int |\epsilon|^{1+\frac{4}{d}}(|\wit{\zeta}_b|+|y\cdot\nabla\wit{\zeta}_b|)\\ \leq &\Gamma_b^{\frac{1}{2}(1-C\eta)}\int_{|y|\leq 2A}|\epsilon|^{1+\frac{4}{d}}\\
	\leq &  C\Gamma_b^{\frac{1}{2}(1-C\eta)}\int_{|y|\leq 2A}|\wit{\epsilon}|^{1+\frac{4}{d}}+\Gamma_b^{1+z_0}\\
	\leq &C\Gamma_b^{\frac{1}{2}(1-C\eta)}\int_{|y|\leq 2A}\left(|I_{N\lambda}\wit{\epsilon}|^{1+\frac{4}{d}}+|(\mathrm{Id}-I_{N\lambda})\wit{\epsilon}|^{1+\frac{4}{d}}\right)+\Gamma_b^{1+z_0}
	\\
	\leq &C\Gamma_b^{\frac{1}{2}(1-C\eta)}A^C\left(\int_{|y|\leq 2A} |I_{N\lambda}\wit{\epsilon}|^{\frac{2d}{d-2}}\right)^{\frac{(d+4)(d-2)}{2d^2}}+\Gamma_b^{\frac{1}{2}(1-C\eta)}A^{\frac{2d}{1-\frac{4}{d}}}\|(\mathrm{Id}-I_{N\lambda})\wit{\epsilon}\|_{L^2}^{\frac{1+\frac{4}{d}}{2}}+\Gamma_b^{1+z_0}\\
	\leq &C\Gamma_b^{\frac{1}{2}(1-C\eta)}A^C\|\nabla I_{N\lambda}\wit{\epsilon}\|_{L^2}^{\frac{d+4}{d}}+\Gamma_b^{1+z_0}\\
	\leq &\Gamma_b^C\|\nabla I_{N\lambda}\wit{\epsilon}\|_{L^2}^2+\Gamma_b^{1+z_0},
	\end{split}
	\end{equation*}
	provided that $a>0$, $\eta>0$ are small enough ($A=e^{\frac{2a\theta(2)}{b}}$).
	
	(4) follows from the classical pointwise bound
	$$ \left| P(y)\left(|\nabla_y^k\wit{\zeta}_b|+|\nabla_y^l\partial_b\wit{\zeta}_b|\right)\right|\leq A^C\Gamma_b^{\frac{1}{2}-C\eta}
	$$
	and (1).
	
	(5) follows from the pointwise bound
	$$ \|F\|_{L^{\infty}}+\|y\cdot\nabla F\|_{L^{\infty}}\leq \frac{C\Gamma_b^{\frac{1}{2}}}{A^{\frac{d}{2}}}
	$$
	and Cauchy-Schwartz inequality.
	%%%%%%%%%%%%%%%%%%%%%%%%%%%%%%%%%%%%%%%%
	
	\textbf{Step 3}: Estimate of terms involving geometric parameters:	
	
	Denote by $\wit{\epsilon}_1=\epsilon_1-\Re\wit{\zeta}_b,\wit{\epsilon_2}=\epsilon_2-\Re\wit{\zeta}_2$, the terms to be estimated are:
	\begin{equation*}
	\begin{split}
	&Term_1=b_s\left[(\wit{\epsilon}_2,\partial_b\Lambda(\Sigma+\Re\wit{\zeta}_b))-(\wit{\epsilon}_1,\partial_b\Lambda(\Theta+\Im\wit{\zeta}_b))\right],\\
	&Term_2=\left(\frac{\lambda_s}{\lambda}+b\right)\left[(\wit{\epsilon}_2,\Lambda^2(\Sigma+\Re\wit{\zeta}_b))-(\wit{\epsilon}_1,\Lambda^2(\Theta+\Im\wit{\zeta}_b))\right],\\
	&Term_3=\frac{x_s}{\lambda}\cdot\left[(\wit{\epsilon}_2,\nabla\Lambda(\Sigma+\Re\wit{\zeta}_b))-(\wit{\epsilon}_1,\nabla\Lambda(\Theta+\Im\wit{\zeta}_b))\right],\\
	&Term_4=\wit{\gamma}_s\left[(\wit{\epsilon}_1,\Lambda(\Sigma+\Re\wit{\zeta}_b))+(\wit{\epsilon}_2,\Lambda(\Theta+\Im\wit{\zeta}_b))\right],\\
	&Term_5=\frac{A_s}{A}\left[(\wit{\epsilon}_2,\Lambda(y\cdot\nabla\chi(\frac{y}{A})\Re\wit{\zeta}_b))-
	(\wit{\epsilon}_1,\Lambda(y\cdot\nabla\chi(\frac{y}{A})\Im\wit{\zeta}_b))\right]
	\end{split}
	\end{equation*}
	\underline{Claim 3:}\\
	\begin{enumerate}
		\item
		$$ |Term_1|+|Term_2|+|Term_3|\leq \delta(\alpha^*)\left(\Xi(s)+\int|I_{N\lambda}\wit{\epsilon}|^2+\int|\wit{\epsilon}|^2e^{-|y|}\right)+\Gamma_b^{1+z_0}.
		$$
		\item
		$$ \left|Term_4-\frac{(\wit{\epsilon}_1,L_+\Lambda^2Q)(\wit{\epsilon}_1,\Lambda Q)}{\|\Lambda Q\|_{L^2}^2}\right|\leq \delta(\alpha^*)\left(\Xi(s)+\int|I_{N\lambda}\wit{\epsilon}|^2+\int|\wit{\epsilon}|^2e^{-|y|}\right)+\Gamma_b^{1+z_0}.
		$$
		\item
		$$ |Term_5|\leq \delta(\alpha^*)\left(\int|I_{N\lambda}\wit{\epsilon}|^2+\int|\wit{\epsilon}|^2e^{-|y|}\right)+\Gamma_b^{1+z_0}.
		$$
	\end{enumerate}
	\underline{Proof:}\\
	For (1), the three terms on the left hand side can be estimated in a similar way by using  Lemma 3.6 and Claim 1.
	\begin{equation*}
	\begin{split}
	|Term_1|\leq & |b_s|\left(\int|\wit{\epsilon}|^2e^{-\kappa |y|}\right)^{\frac{1}{2}}\\
	\leq &C\left(\Xi(s)+\int|\nabla I_{N\lambda}\epsilon|^2+\int |\epsilon|^2e^{-|y|}+\Gamma_b^{1-C\eta}\right)\left(\int |\wit{\epsilon}|^2e^{-|y|}\right)^{\frac{1}{2}}\\
	\leq &C\left(\Xi(s)+\int|\nabla I_{N\lambda}\wit{\epsilon}|^2+\int |\wit{\epsilon}|^2e^{-|y|}+\Gamma_b^{1-C\eta}\right)\left(\int |\wit{\epsilon}|^2e^{-|y|}\right)^{\frac{1}{2}}\\
	\leq &\delta(\alpha^*)\left(\Xi(s)+\int|\nabla I_{N\lambda}\wit{\epsilon}|^2+\int|\wit{\epsilon}|^2e^{-|y|}\right)+\Gamma_b^{1+z_0}
	\end{split}
	\end{equation*}
	where in the last step, we write the term $$\Gamma_b^{1-C\eta}\left(\int |\wit{\epsilon}|^2e^{-|y|}\right)^{\frac{1}{2}}=\Gamma_b^{1-C\eta-\eta'}\left(\Gamma_b^{\eta'}\left(\int |\wit{\epsilon}|^2e^{-|y|}\right)^{\frac{1}{2}}\right)$$	
	and use the elementary inequality $XY\leq\frac{X^2+Y^2}{2}$.
	
	For (2), the left hand side can be estimated as
	\begin{equation*}
	\begin{split}
	&\left|\wit{\gamma}_s-\frac{(\wit{\epsilon_1},L_+\Lambda^2 Q)}{\|\Lambda Q\|_{L^2}^2}\right|\left|(\wit{\epsilon_1},\Lambda(\Sigma+\Re\wit{\zeta}_b))+(\wit{\epsilon_2},\Lambda(\Theta+\Im\wit{\zeta}_b))\right|\\
	+&\left|\frac{(\wit{\epsilon_1},L_+\Lambda^2 Q)}{\|\Lambda Q\|_{L^2}^2}\right|\left|(\wit{\epsilon_1},\Lambda(\Sigma-Q+\Re\wit{\zeta}_b))+(\wit{\epsilon_2},\Lambda(\Theta+\Im\wit{\zeta}_b))\right|.
	\end{split}
	\end{equation*}
	From  Lemma 3.6,
	\begin{equation*}
	\begin{split}
	\left|\wit{\gamma}_s-\frac{(\wit{\epsilon_1},L_+\Lambda^2 Q)}{\|\Lambda Q\|_{L^2}^2}\right|\leq & \left|\wit{\gamma}_s-\frac{(\epsilon_1,L_+\Lambda^2 Q)}{\|\Lambda Q\|_{L^2}^2}\right|+\frac{|(\Re\wit{\zeta}_b,L_+^2\Lambda^2 Q)|}{\|\Lambda Q\|_{L^2}^2}\\
	\leq &\delta(\alpha^*)\left(\int |\nabla I_{N\lambda}\wit{\epsilon}|^2+\int |\wit{\epsilon}|^2e^{-|y|}\right)^{\frac{1}{2}}+\Gamma_b^{\frac{1}{2}+z_0'}+C\Xi(s)
	\end{split}
	\end{equation*}
	since
	$$ \int |\wit{\zeta}_b||L^+\Lambda^2 Q|\leq C\int|\wit{\zeta}_b|e^{-\kappa|y|}\leq \Gamma_b^{\frac{1}{2}+z_0'},
	$$
	thanks to the property of $\zeta_b$, (see \cite{MR06JAMS}). Now (2) follows from (4) of Claim 2.
	
	For (3), we note that $\left|\frac{A_s}{A}\right|\leq C\left|\frac{b_s}{b^2}\right|$ and
	$$ \|\Lambda(y\cdot\nabla\left(\frac{y}{A}\right)\zeta_b)\|_{L^{\infty}}\leq A^C\Gamma_b^{1-C\eta},
	$$
	thus
	\begin{equation*}
	\begin{split}
	|Term_5|\leq & C\left|\frac{b_s}{b^2}\right|A^C\Gamma_b^{1-C\eta}\left(\int_{|y|\leq 2A}|\wit{\epsilon}|^2\right)^{\frac{1}{2}}\\
	\leq &C\left|\frac{b_s}{b^2}\right|A^C\Gamma_b^{1-C\eta}\left[A\left(\int |\nabla I_{N\lambda}\wit{\epsilon}|^2+\int|\wit{\epsilon}|^2e^{-|y|}\right)^{\frac{1}{2}}+\Gamma_b^{\frac{1}{2}+z_0}\right]\\
	\leq &\delta(\alpha^*)\left(\int|I_{N\lambda}\wit{\epsilon}|^2+\int|\wit{\epsilon}|^2e^{-|y|}\right)+\Gamma_b^{1+z_0}
	\end{split}
	\end{equation*}
	by writing that
	$$\|\wit{\epsilon}\|_{L^2(|y|\leq 2A)}\leq \|I_{N\lambda}\wit{\epsilon}\|_{L^2(|y|\leq 2A)}+\|(\mathrm{Id}-I_{N\lambda})\wit{\epsilon}\|_{L^2}
	$$
	and using (1) of Claim 1.

	%%%%%%%%%%%%%%%%%%%%%%%%%%%%%%%%%%%%%%%%%%%%%%%%%%%%%%%%%%%%
	\textbf{Step 4}: Estimate of degenerate scalar products.\\
	\underline{Claim 4}:
	\begin{enumerate}
		\item
		$$ (\wit{\epsilon_1},Q)^2\leq \delta(\alpha^*)\left(\int|\nabla I_{N\lambda}\wit{\epsilon}|^2+\int|\wit{\epsilon}|^2e^{-|y|}+\Xi(s)\right)+\Gamma_b^{1+z_0}
		$$
		\item $$(\wit{\epsilon}_2)^2\leq \delta(\alpha^*)\left(\int|\nabla I_{N\lambda}\wit{\epsilon}|^2+\int|\wit{\epsilon}|^2e^{-|y|}\right)+\Gamma_b^{1+z_0}.
		$$
		\item
		\begin{equation*}
		\begin{split}
		&(\wit{\epsilon_1},|y|^2Q)^2+(\wit{\epsilon}_1,yQ)^2+(\wit{\epsilon}_2,\Lambda Q)^2+(\wit{\epsilon_2},\Lambda^2 Q)^2\\ \leq &\delta(\alpha^*)\left(\int|\nabla I_{N\lambda}\wit{\epsilon}|^2+\int|\wit{\epsilon}|^2e^{-|y|}\right)+\Gamma_b^{1+z_0}.
		\end{split}
		\end{equation*}
	\end{enumerate}
	\underline{Proof:}\\
	We first indicate that $\Xi(s)\leq\delta(\delta^*)$. Essentially,
	$$ \Xi(s)=\frac{\lambda^2(s)}{2}\int |J_{N(s)}G(0,x)|^2dx
	$$
	with $J_N=\mathrm{Id}-I_{N}$.
	Recall that
	$$G(0,x)=\frac{1}{\lambda(0)^{\frac{d}{2}}}(Q_{b()}+g(0))\left(\frac{x-x(0)}{\lambda(0)}\right)e^{-i\gamma(0)},
	$$
	we estimate
	\begin{equation*}
	\begin{split}
	\frac{\lambda^2(s)}{2}\|J_{N(s)}G(0)\|_{L^2}^2\leq &\frac{\lambda(s)^2}{\lambda(0)^2}\|\nabla g(0)\|_{L^2}^2+\|J_{N(s)\lambda(0)}(\nabla Q_{b(0)})\|_{L^2}^2\\
	\leq & C\Gamma_{b(0)}^{\frac{3}{4}}+\frac{1}{N(s)\lambda(s)}\frac{\lambda(s)}{\lambda(0)}\\
	\leq & \delta(\alpha^*)
	\end{split}
	\end{equation*}
	thanks to the bootstrap assumption.
	
	Now (1) follows from the following estimate induced by almost conservation of energy:
	\begin{equation}\label{tildeenergy}
	(\wit{\epsilon}_1,Q)^2\leq \delta(\alpha^*)\left(\int|\nabla I_{N\lambda}\wit{\epsilon}|^2+\int |\wit{\epsilon}|^2e^{-|y|}+\Xi(s)\right)+\Gamma_b^{1+z_0}+(\lambda^2E(I_Nu))^2
	\end{equation}
	Indeed, we have already seen in Step 3 that
	$$ \int |\wit{\zeta}_b|^2e^{-\kappa|y|}\leq \Gamma_b^{\frac{1}{2}+z_0}.
	$$
	Thus we could replace the left hand side by $(\epsilon_1,Q)^2$. From the almost conservation of energy
	Lemma 3.2, we have
	\begin{equation*}
	\begin{split}
	2(\epsilon_1,Q)=-&2(\epsilon_1,\Sigma-Q+b\Lambda\Theta)-2(\epsilon,\Theta-b\Lambda\Sigma)+2(\epsilon_1,\Re\Psi)+2(\epsilon_2,\Im\Psi)\\
	+&2\Xi(s)+O\left(\int|\nabla I_{N\lambda}\epsilon|^2+\int|\epsilon|^2e^{-|y|}+\Gamma_b^{1-C\eta}\right)
	\end{split}
	\end{equation*}
	Then we conclude by the estimates
	$$ |\Re(\epsilon,Q_b-Q+ib\Lambda Q_b)|\leq \delta(\alpha^*)\left(\int |\nabla I_{N\lambda}\epsilon|^2+\int|\epsilon|^2e^{-|y|}\right)^{\frac{1}{2}}+\Gamma_b^{\frac{1}{2}+z_0},
	$$
	\begin{equation*}
	\begin{split}
	\left|\int |\epsilon||\Psi|\right|^2\leq &\Gamma_b^{1-C\eta}\left(\int_{R_b^-\leq|y|\leq R_b}|\epsilon|\right)^2\\
	\leq &\Gamma_b^{1-C\eta}\left(\int |\nabla I_{N\lambda}\epsilon|^2+\int|\epsilon|^2e^{-|y|}\right)^{\frac{1}{2}}+\Gamma_b^{\frac{1}{2}+z_0}
	\end{split}
	\end{equation*}
	and the Claim 1.

	(2) follows from the almost conservation of momentum
	\begin{align*}
	-2(\var_2,\nabla Q)=&\lambda P(I_Nu)-\Big(P\big(I_{N\lambda}(Q_b+\var)\big)-P(Q_b+I_{N\lambda}\var)\Big)\\
	&-2(\var_2,\nabla(Q-\Sigma))-2(\var_2,({\rm Id}-I_{N\lambda})\nabla\Sigma)\\
	&-2\int I_{N\lambda}\var_1\nabla\Theta-2\int I_{N\lambda}\var_1\nabla I_{N\lambda}\var_2.
	\end{align*}
	Again, we could replace the left hand side by $(\epsilon_2,\nabla Q)^2$. The desired estimate follows from the same manipulation as for (1). Finally, (3) follows from orthogonality condition imposed by $\epsilon$ and the small deformation of $Q$ to $Q_b$.

	\underline{Claim 5}:
	\begin{equation*}
	\begin{split}
	&|(\epsilon_1,\wit{L})|+|(\epsilon_2,\wit{K})|+|(\wit{\epsilon_1},\Lambda(Q^{\frac{4}{d}}\Re\wit{\zeta}_b))|+|(\wit{\epsilon_2},\Lambda(Q^{\frac{4}{d}}\Im\wit{\zeta}_b))|\\
	\leq & \delta(\alpha^*)\left(\int|\nabla I_{N\lambda}\wit{\epsilon}|^2+\int|\wit{\epsilon}|^2e^{-|y|}\right)+\Gamma_b^{1+z_0}.
	\end{split}
	\end{equation*}
	\underline{Proof:}\\
	Using the expression of $\wit{L}$, the property of $Q_b,\wit{\zeta}_b$, we estimate
	\begin{equation*}
	\begin{split}
	|(\epsilon_1,\wit{L})|=&\left|\int \epsilon_1\left[\left(\left(\frac{4\Sigma^2}{d|Q_b|^2}+1\right)|Q_b|^{\frac{4}{d}}-\left(1+\frac{4}{d}\right)Q^{\frac{4}{d}}\right)\Lambda\Re\wit{\zeta}_b+\frac{4\Sigma\Theta}{d|Q_b|^{2}}|Q_b|^{\frac{4}{d}}\Lambda\Im\wit{\zeta}_b\right]\right|\\
	\leq & C\left(\int |\epsilon|^2e^{-\frac{2}{d}(1-C\eta)\frac{\theta(b|y|)}{|b|}}\right)^{\frac{1}{2}}\Gamma_b^{1+\frac{C}{d}-C\eta}\\
	\leq &\delta(\alpha^*)\left(\int|\nabla I_{N\lambda}\wit{\epsilon}|^2+\int|\wit{\epsilon}|^2e^{-|y|}\right)+\Gamma_b^{1+z_0}
	\end{split}
	\end{equation*}
	Similarly for $|(\epsilon_2,\wit{K})|$. For the last two terms, we estimate, for example
	\begin{equation*}
	\begin{split}
	|(\wit{\epsilon}_1,\Lambda(Q^{\frac{4}{d}}\Re\wit{\zeta}_b))|\leq & \int|\wit{\epsilon}|(|\wit{\zeta}_b|+|\nabla\wit{\zeta}_b|)e^{-\kappa|y|}\\
	\leq &\left(\int |\wit{\epsilon}|^2e^{-\kappa|y|}\right)^{\frac{1}{2}}\Gamma_b^{\frac{1}{2}+z_0'}\\
	=&\left(\Gamma_b^{z_0'}\int |\wit{\epsilon}|^2e^{-\kappa|y|}\right)^{\frac{1}{2}}\Gamma_b^{\frac{1+z_0'}{2}}\\
	\leq &\delta(\alpha^*)\left(\int|\nabla I_{N\lambda}\wit{\epsilon}|^2+\int|\wit{\epsilon}|^2e^{-|y|}\right)+\Gamma_b^{1+z_0}.
	\end{split}
	\end{equation*}
	%%%%%%%%%%%%%%%%%%%%%%%%%%%%%%%%%%%%%%%%

	\textbf{Step 5}: Estimate of principal terms.
	
	\underline{Claim 6:}
	\begin{equation}\label{Claim6}
	\begin{split}
	&H(I_{N\lambda}\epsilon-\wit{\zeta}_b,I_{N\lambda}\epsilon-\wit{\zeta}_b)-\frac{(\wit{\epsilon_1},L_+\Lambda^2 Q)(\wit{\epsilon_1},\Lambda Q)}{\|\Lambda Q\|_{L^2}^2}-\left[(\Re\wit{\zeta}_b,\Lambda\Re F)+(\Im\wit{\zeta}_b,\Lambda\Im F)\right]\\
	\geq & c_1\left(\int|\nabla I_{N\lambda}\wit{\epsilon}|^2+\int|\wit{\epsilon}|^2e^{-|y|}+\Gamma_b\right)-\delta(\alpha^*)\Xi(s)
	\end{split}
	\end{equation}
	\underline{Proof:}\\
	Indeed, we may replace $I_{N\lambda}-\wit{\zeta}_b$ by $I_{N\lambda}\wit{\epsilon}$ by adding an error $\Gamma_b^{1+z_0}$ which does not change the type of the desired estimate. From the spectral property and Claim 4, we have
	\begin{equation*}
	\begin{split}
	&H(I_{N\lambda}\wit{\epsilon},I_{N\lambda}\wit{\epsilon})-\frac{(\wit{\epsilon_1},L_+\Lambda^2 Q)(\wit{\epsilon_1},\Lambda Q)}{\|\Lambda Q\|_{L^2}^2}\geq c_0\left(\int |I_{N\lambda}\wit{\epsilon}|^2+\int |I_{N\lambda}\wit{\epsilon}|^2e^{-|y|}\right)\\
	-&\frac{1}{c_0}\left[
	(\wit{\epsilon}_1,Q)^2+(\wit{\epsilon}_1,|y|^2Q)^2+(\wit{\epsilon}_1,yQ)^2+(\wit{\epsilon}_2,\Lambda Q)^2+(\wit{\epsilon}_2,\Lambda^2 Q)^2+(\wit{\epsilon}_2,\nabla Q)^2\right]\\
	\geq & c_1\left(\int |I_{N\lambda}\wit{\epsilon}|^2+\int |\wit{\epsilon}|^2e^{-|y|}\right)-\delta(\alpha^*)\Xi(s)-\Gamma_b^{1+z_0}
	\end{split}
	\end{equation*}
	Finally, we conclude by the estimate proved in \cite{MR06JAMS}:
	$$-\left[(\Re\wit{\zeta}_b,\Lambda\Re F)+(\Im\wit{\zeta}_b,\Lambda\Im F)\right]>c\Gamma_b.
	$$
	%%%%%%%%%%%%%%%%%%%%%%%%%%%%%%%%%%%%%%%%
	\textbf{Step 6}: Estimate of original reminder terms and conclusion.
	
	The classical remainder term $\mathcal{R}emainder$ has been already estimated in the proof of Lemma 3.4, while the term $\Upsilon_{N\lambda}$ can be also bounded easily by $\Gamma_b^{1+z_0}$.
	Combining Step 1 to Step 5, we obtain the desired estimate.

\end{proof}
Next, we need to control the $L^2$ type of term $\int_A^{2A}|\var|^2$ in \eqref{equ:f1sdef}. This is achieved by computing the flux of $L^2$ norm escaping the radiative zone. To do it, we introduce a radial nonnegative cutoff function $\psi(r)$ such that
\begin{align*}
\psi(r)= \begin{cases} 0 \quad r\leq\frac12\\
1\quad r\geq3, \end{cases}
\frac14\leq\psi'(r)\leq\frac12\quad\text{for}~1\leq r\leq2,~\psi'(r)\geq0.
\end{align*}
Let
$$\psi_A(s,r)=\psi\big(\tfrac{r}{A(s)}\big),$$
with $A(s)$ being given by \eqref{equ:atdef}, and so
\begin{align*}
\begin{cases}
\psi_A(r)=0~\text{for}~0\leq r\leq\frac{A}{2},\\
\frac{1}{4A}\leq\psi_A'(r)\leq\frac{1}{2A}~\text{for}~A\leq r\leq2A,\\
\psi_A(r)=1~\text{for}~r\geq3r,\\
\psi_A'(r)\geq0,~0\leq\psi_A(r)\leq1.
\end{cases}
\end{align*}

Moreover, we restrict the freedom on the choice of the parameters $(\eta,a)$ by assuming $a>C\eta$.

\begin{lemma}[$L^2$ dispersion at infinity in space]\label{lem:l2dispinf}
There holds for some universal constants $C,c_3>0$ and $s$ large enough:
\begin{equation}\label{equ:psiavar}
\Big\{\int \psi_A|\var|^2\Big\}_s\geq c_3b\int_A^{2A}|\var|^2-\frac{C}{b^2}\Xi(s)-\Gamma_b^{1+Ca}-\Gamma_b^\frac{a}{2}\int|\nabla I_{N(s)\lambda(s)}\var|^2.
\end{equation}
\end{lemma}

\begin{proof}

%First, it is easy to check that $I_{N(t)}u(t)=I_Nu$ solves
%\begin{equation}\label{equ:intutsol}
%i\pa_t(I_Nu)+\Delta(I_Nu)=i\tilde{I}_Nu-I_N\big[|u|^\frac4du\big]
%\end{equation}
%where
%\begin{equation}\label{equ:tilindef}
%\widehat{\tilde{I}_Nu}(\xi)=-\frac{N_t}{N}\tilde{m}_N(\xi)\widehat{I_Nu}%(\xi),~~\tilde{m}_N(\xi)=\frac{\xi}{N}\cdot\frac{\nabla m_N(\xi)}{m_N(\xi)}.
%\end{equation}
%Hence,
%\begin{align*}
%&\frac12\frac{d}{ds}\int\psi\Big(\frac{x-x(t)}{A\lambda(t)}\Big)|I_Nu|^2\\
%=&\frac12\int\frac{d}{ds}\Big[\psi\Big(\frac{x-x(t)}{A\lambda(t)}\Big)\Big]|I_Nu|^2+{\rm Re}\int
%\psi\Big(\frac{x-x(t)}{A\lambda(t)}\Big)\pa_s\big(I_Nu\big)\overline{I_Nu}\\
%=&-\frac12\int\Big(\frac{\lambda_s}{\lambda}y+\frac{x_s}{\lambda}\Big)\cdot%\nabla\psi_A\big|I_{N\lambda}(Q_b+\var)\big|^2
%\end{align*}

Take the inner product of \eqref{equ:var1} with $\psi_A\var_1$ and of \eqref{equ:var2} with $\psi_A\var_2$ and integrate by parts. Note that the supports of $(Q_b,\Psi_b)$ and $\psi_A$ are disjoint. Then, we obtain
\begin{align}\nonumber
\frac12\Big\{\int\psi_A|\var|^2\Big\}_s=&\frac12\int\frac{\pa\psi_A}{\pa s}|\var|^2+\frac{b}{2}\int y\cdot\psi_A|\var|^2+{\rm Im}\Big(\int\nabla\psi_A\cdot\nabla\var\bar{\var}\Big)\\\label{equ:dsvar2est}
&-\frac12\Big(\frac{\lambda_s}{\lambda}+b\Big)\int y\cdot\nabla\psi_A|\var|^2-\frac12\frac{x_s}{\lambda}\int\nabla\psi_A|\var|^2.
\end{align}
First observe from the choice of $\psi:$
\begin{equation}\label{equ:phivar2}
10\int\psi'\big(\tfrac{y}{A}\big)|\var|^2\geq\frac{1}{A}\int y\cdot\nabla\psi\big(\tfrac{y}{A}\big)|\var|^2\geq\frac1{10}\int\psi'\big(\tfrac{y}{A}\big)|\var|^2\geq\frac1{40}\int_A^{2A}|\var|^2.
\end{equation}
The main term is
\begin{equation}\label{equ:maintermvar1}
\frac{b}{2}\int y\cdot\psi_A|\var|^2\geq\frac1{20}\int\psi'\big(\tfrac{y}{A}\big)|\var|^2.
\end{equation}
Using \eqref{equ:atdef} and \eqref{equ:lambsest}, we get
\begin{align}\label{equ:secmaitre}
&\int\frac{\pa\psi_A}{\pa s}|\var|^2=-\frac{A_s}{A^2}\int y\cdot\nabla\psi\big(\tfrac{y}{A}\big)|\var|^2=2a\frac{b_s}{Ab^2}\int y\cdot\nabla\psi\big(\tfrac{y}{A}\big)|\var|^2\\\nonumber
\geq&\frac{a}{Ab^2}\Big[c_0\Big(\Xi(s)+\int|\nabla I_{N(s)\lambda(s)}\var(s)|^2+\int|\var(s)|^2e^{-|y|}\Big)-\Gamma_{b(s)}^{1-C\eta}\Big]\int\psi'\big(\tfrac{y}{A}\big)|\var|^2.
\end{align}
Next,
by Cauchy-Schwartz inequality, we have
\begin{align*}
&\Big|\int\nabla\psi_A\cdot\nabla\var\bar{\var}\Big|\\
\leq&\Big|\int\nabla\psi_A\cdot\nabla(I_{N\lambda}\var)\bar{\var}\Big|
+\Big|\int\nabla\psi_A\cdot\nabla\big(({\rm Id}-I_{N\lambda})\var\big)\bar{\var}\Big|\\
\leq&\frac{b}{100}\int\psi'\big(\tfrac{y}{A}\big)|\var|^2+\Gamma_b^\frac{a}{2}\int|\nabla I_{N\lambda}\var|^2+\Big|\int\nabla\psi_A\cdot\nabla\big(({\rm Id}-I_{N\lambda})\var\big)\bar{\var}\Big|.
\end{align*}
On the other hand, by duality, Bernstein's inequaltiy, and interpolation, we estimate
\begin{align*}
\Big|\int\nabla\psi_A\cdot\nabla\big(({\rm Id}-I_{N\lambda})\var\big)\bar{\var}\Big|\leq&\big\|({\rm Id}-I_{N\lambda})\var\big\|_{\dot{H}^\frac12}\|\nabla\psi_A\bar{\var}\|_{\dot{H}^\frac12}\\
\lesssim&\frac{(N\lambda)^{\frac12-s}}{A}\|\var\|_{\dot{H}^s}\|\var\|_{\dot{H}^\frac12}\\
\lesssim&\Gamma_b^\frac{a}2\lambda(t)^{\frac{1-\beta}{\beta}(s-\frac12)}
\|\var\|_{L^2}^{1-\frac1{2s}}\|I_{N\lambda}\var\|_{\dot{H}^1}^{1+\frac{1}{2s}}\\
\leq&\Gamma_b^{10}.
\end{align*}
Using \eqref{equ:gammaxsest}, we derive
\begin{equation}\label{equ:xsterm}
\Big|\frac{x_s}{\lambda}\int\nabla\psi_A|\var|^2\Big|\leq\frac{C}{A}\int\psi'\big(\tfrac{y}{A}\big)|\var|^2\leq\Gamma_b^\frac{a}2\int\psi'\big(\tfrac{y}{A}\big)|\var|^2.
\end{equation}
By the same way, we have by  \eqref{equ:lambsest}
\begin{align}\label{equ:lamterm}
&\Big|\Big(\frac{\lambda_s}{\lambda}+b\Big)\int y\cdot\nabla\psi_A|\var|^2\Big|\\\nonumber\leq& C\Big(\Xi(s)+\int|\nabla I_{N(s)\lambda(s)}\var(s)|^2+\int|\var(s)|^2e^{-|y|}+\Gamma_{b(s)}^{1-C\eta}\Big)\int\psi'\big(\tfrac{y}{A}\big)|\var|^2.
\end{align}

\end{proof}

\begin{corollary}[Lyapounov functional]\label{cor:jsest}
For some universal constant $C>0$ anf for $s$ large, the following holds:
\begin{align}\nonumber
\{\mathcal{J}\}_s\leq&-Cb\Big(\Gamma_b+\Xi(s)+\int|\nabla I_{N(s)\lambda(s)}\tilde{\var}|^2+\int|\tilde{\var}(s)|^2e^{-|y|}+\int_A^{2A}|\var(s)|^2\Big)\\\label{equ:jsest}
&+C\frac{\Xi(s)}{b^2},
\end{align}
with
\begin{align}\nonumber
\mathcal{J}(s)=&\Big(\int |Q_b|^2-\int Q^2\Big)+2(\var_1,\Sigma)+2(\var_2,\Theta)+\int(1-\psi_A)|\var|^2\\\label{equ:jsdef}
&-c\Big(b\tilde{f}_1(b)-\int_0^b\tilde{f}_1(v)\;dv+b\big[{\rm Re}(\var_2,\Lambda\tilde{\zeta})-{\rm Im}(\var_1,\Lambda\tilde{\zeta})\big]\Big),
\end{align}
where $c>0$ denotes some small enough universal constant and
\begin{equation}\label{equ:tildef1def}
\tilde{f}_1(b)=\frac{b}{4}\|yQ_b\|_{L^2}^2+\frac12{\rm Im}\Big(\int(y\cdot\nabla\tilde{\zeta})\bar{\tilde{\zeta}}\Big).
\end{equation}
\end{corollary}

Finally, by the same argument as Subsection 4.3 in \cite{CR09}, we conclude the proof of the bootstrap Lemma \ref{lem:bootstrap}.

\subsection{End proof of Proposition \ref{prop:main}}

Now, we are in position to proving Proposition \ref{prop:main}. The proof is the same as in \cite{CR09}. The main difference is the step 4: Strong $L^2$ convergence of excess mass outside the blowup point and Step 5: nonconcentration of the $L^2$ norm at the blowup point. We are reduced to show the following lemma.

\begin{lemma}\label{lem:finallemm}
Let $R>0$. Let $x(T)=\lim\limits_{t\to T}x(t)$. Then, there exists $u^\ast\in L^2$ such that
\begin{equation}\label{equ:l2convetout}
u(t)\to u^\ast~\text{in}~L^2(\R^d\backslash\{|x-x(T)|\leq R\})~\text{as}~t\to T.
\end{equation}
Furthermore, there holds
\begin{equation}\label{equ:uastl2est}
\int|u^\ast|^2=\lim_{t\to T}\int\chi_{R(t)}|I_{N(t)}u(t)|^2,~R(t)=10A(t)\lambda(t),
\end{equation}
 with
\begin{equation}
\chi_R(x)=\chi\big(\tfrac{x-x(t)}{R}\big),~\chi(x)=\begin{cases} 0\quad |x|\leq1\\
1\quad|x|\geq2.
\end{cases}
\end{equation}
\end{lemma}

\begin{proof}
First, we prove the claim.
For $R>0$, let $w_R(t,x)=\chi_R(x)\big[I_{N(t)}u(t,x)\big]$
then, $w_R$ solves
\begin{equation}\label{equ:wrequ}
i\pa_tw_R+\Delta w_R=i\chi_R\tilde{I}_Nu+2\nabla\chi_R\cdot\nabla I_Nu+\Delta\chi_RI_Nu-\chi_RI_N\big(|u|^\frac4du\big).
\end{equation}
where
\begin{equation}\label{equ:tilindef}
\widehat{\tilde{I}_Nu}(\xi)=-\frac{N_t}{N}\tilde{m}_N(\xi)\widehat{I_Nu}(\xi),~~\tilde{m}_N(\xi)=\frac{\xi}{N}\cdot\frac{\nabla m_N(\xi)}{m_N(\xi)}.
\end{equation}
It follows from \cite{CR09} that
\begin{align*}
&\Big\|\int_0^te^{i(t-\tau)\Delta}\big(i\chi_R\tilde{I}_Nu+2\nabla\chi_R\cdot\nabla I_Nu+\Delta\chi_RI_Nu\big)(\tau)\;d\tau\Big\|_{L_t^\infty L_x^2([0,T]\times\R^d)}<+\infty.
\end{align*}
We only need to estimate the nonlinear term
\begin{align}\label{equ:nonestred}
\Big\|\int_0^te^{i(t-\tau)\Delta}\chi_RI_N\big(|u|^\frac4du\big)\;d\tau\Big\|_{L_t^\infty L_x^2([0,T]\times\R^d)}<+\infty.
\end{align}
And
\begin{align*}
&\Big\|\int_0^te^{i(t-\tau)\Delta}\chi_RI_N\big(|u|^\frac4du\big)\;d\tau\Big\|_{L_t^\infty L_x^2([0,T]\times\R^d)}\leq C\big\|\chi_RI_N\big(|u|^\frac4du\big)\big\|_{L_t^2L_x^\frac{2d}{d+2}([0,T]\times\R^d)}\\
\lesssim&\big\|\chi_R\big(|I_Nu|^\frac4dI_Nu\big)\big\|_{L_t^2L_x^\frac{2d}{d+2}([0,T]\times\R^d)}+\big\||I_Nu|^\frac4dI_Nu-I_N(|u|^\frac4du)\big\|_{L_t^2L_x^\frac{2d}{d+2}([0,T]\times\R^d)}.
\end{align*}
On one hand,
\begin{align*}
\big\|\chi_R\big(|I_Nu|^\frac4dI_Nu\big)\big\|_{L_t^2L_x^\frac{2d}{d+2}([0,T]\times\R^d)}^2=&\int_0^T\big\|\chi_R^\frac{d}{d+4}I_Nu\big\|_{L_x^\frac{2(d+4)}{d+2}}^\frac{2(d+4)}{d}\;dt\\
\lesssim&\int_0^T\big\|\chi_\frac{R}{2}I_Nu\big\|_{L_x^\frac{2(d+4)}{d+2}}^\frac{2(d+4)}{d}\;dt\\
\lesssim&\int_0^T\Big\|\nabla\big(\chi_\frac{R}{2}I_Nu\big)\Big\|_{L^2}^2<+\infty,
\end{align*}
where the last inequality follows from (4.49) in \cite{CR09}.
On the other hand, we have by Lemma \ref{lem:xianchafa} and \eqref{equ:dius}
\begin{align*}
&\big\||I_Nu|^\frac4dI_Nu-I_N(|u|^\frac4du)\big\|_{L_t^2L_x^\frac{2d}{d+2}([0,T]\times\R^d)}^2\\
\lesssim&\sum_{k=k_0}^{+\infty}\sum_{j=1}^{J_k}\big\||I_Nu|^\frac4dI_Nu-I_N(|u|^\frac4du)\big\|_{L_t^2L_x^\frac{2d}{d+2}([\tau_k^j,\tau_k^{j+1}]\times\R^d)}^2\\
\lesssim&\sum_{k=k_0}^{+\infty}\sum_{j=1}^{J_k}\lambda(t_k)^{\frac2\beta\min\{1,\frac4d\}s_-\frac{8}{d}}\\
\lesssim&\sum_{k=k_0}^{+\infty}k\lambda(t_k)^{\frac2\beta(\min\{1,\frac4d\}s_-\frac{1-s}{s})-\frac{8}{d}}\\
\lesssim&\sum_{k=k_0}^{+\infty}k2^{-\frac{2k}\beta(\min\{1,\frac4d\}s_-\frac{1-s}{s})+\frac{8k}{d}}<+\infty
\end{align*}
provided that $\frac{1}{\beta}>\frac{4s}{\min\{4,d\}s^2-(1-s)}.$
Then, by the standard argument as in \cite{CR09}, we obtain \eqref{equ:l2convetout}.

Next, we turn to show \eqref{equ:uastl2est}. Observer that $I_Nu$ solves
\begin{equation}
i\pa_t(I_Nu)+\Delta(I_Nu)=i\tilde{I}_Nu-I_N\big(|u|^\frac{4}{d}u\big).
\end{equation}
We then compute the flux of $L^2$-norm:
\begin{align*}
&\frac12\frac{d}{d\tau}\int\chi\Big(\frac{x-x(\tau)}{R(t)}\Big)|I_Nu(\tau)|^2\;dx\\
=&\frac{1}{R(t)}{\rm Im}\Big(\int\nabla\chi\Big(\frac{x-x(\tau)}{R(t)}\Big)\cdot\nabla I_Nu\overline{I_Nu}\Big)-\frac{x_t}{2R(t)}\cdot\int\nabla\chi\Big(\frac{x-x(\tau)}{R(t)}\Big)|I_Nu(\tau)|^2\\
&+{\rm Re}\Big(\int \chi\Big(\frac{x-x(\tau)}{R(t)}\Big)\tilde{I}_Nu\overline{I_Nu}\Big)+{\rm Im}\Big(\chi\Big(\frac{x-x(\tau)}{R(t)}\Big)\big(I_Nu|I_Nu|^\frac4d-I_N(u|u|^\frac4d)\big)\overline{I_Nu}\Big)
\end{align*}
and integrate from $t\to T.$ We obtain
\begin{align}\nonumber
&\Big|\int\chi\Big(\frac{x-x(T)}{R(t)}\Big)|u^\ast|^2-\int\chi\Big(\frac{x-x(t)}{R(t)}\Big)|I_Nu(t)|^2\Big|\\\nonumber
\lesssim&\frac{1}{A(t)\lambda(t)}\int_t^T\|\nabla I_{N(\tau)}u(\tau)\|_{L^2}\;d\tau+\frac{1}{A(t)\lambda(t)}\int_t^T\big|\tfrac{x_s}{\lambda}\big|\frac{d\tau}{\lambda(\tau)}+\int_t^T\Big|\int\chi\Big(\frac{x-x(\tau)}{R(\tau)}\Big)\tilde{I}_Nu\overline{I_Nu}\Big|\\\label{equ:inuuuspe}
&+\int_t^T\int\Big|(I_Nu|I_Nu|^\frac4d-I_N(u|u|^\frac4d)\big)\overline{I_Nu}\Big|.
\end{align}
By the same argument as in\cite{CR09}, we obtain
\begin{align*}
&\frac{1}{A(t)\lambda(t)}\int_t^T\|\nabla I_{N(\tau)}u(\tau)\|_{L^2}\;d\tau+\frac{1}{A(t)\lambda(t)}\int_t^T\big|\tfrac{x_s}{\lambda}\big|\frac{d\tau}{\lambda(\tau)}+\int_t^T\Big|\int\chi\Big(\frac{x-x(\tau)}{R(\tau)}\Big)\tilde{I}_Nu\overline{I_Nu}\Big|\\
&\to0\quad\text{as}\quad t\to T.
\end{align*}
Hence, we only need to show
$$\int_t^T\int\Big|(I_Nu|I_Nu|^\frac4d-I_N(u|u|^\frac4d)\big)\overline{I_Nu}\Big|\to 0\quad\text{as}\quad t\to T.$$
Indeed, by \eqref{equ:xianchafa}, we obtain
\begin{align*}
&\int_t^T\int\Big|(I_Nu|I_Nu|^\frac4d-I_N(u|u|^\frac4d)\big)\overline{I_Nu}\Big|\\
\lesssim&\sum_{k=k_t}^{+\infty}\sum_{j=1}^{J_k}\|I_Nu\|_{L_t^2L_x^\frac{2d}{d-2}([\tau_k^j,\tau_k^{j+1}]\times\R^d)}\big\|I_Nu|I_Nu|^\frac4d-I_N(|u|^\frac4du)\big\|_{L_t^2L_x^\frac{2d}{d+2}([\tau_k^j,\tau_k^{j+1}]\times\R^d)}\\
\lesssim&\sum_{k=k_t}^{+\infty}k\lambda(t_k)^{\frac1\beta(\min\{1,\frac4d\}s-\frac{1-s}{s})-\min\{1,\frac4d\}}\\
\lesssim&\sum_{k=k_t}^{+\infty}k2^{-k(\frac1\beta(\min\{1,\frac4d\}s-\frac{1-s}{s})-\min\{1,\frac4d\})}\to0\quad\text{as}\quad k_t\to+\infty.
\end{align*}

\end{proof}

\subsection*{Acknowledgements} J. Zheng was supported by NSFC Grants 11771041, 11831004.  Part of this work was completed while J. Zheng was supported by the the European Research Council, ERC-2014-
CoG, project number 646650 Singwave at Universit\'e de Nice.

\begin{center}

\end{center}


\begin{thebibliography}{99}
%\addcontentsline{toc}{section}{References}


\bibitem{BL} H. Berestycki, P. L. Lions, Nonlinear scalar field equations. I Existence
of a ground state, Arch. Rational Mech. Anal. 82:4 (1983), 313-345.



\bibitem{CaW} T. Cazenave and F. Weissler, The Cauchy problem for the critical nonlinear Schr\"odinger equation in $H^s$.
 Nonlinear Anal., 14(1990),
807-836.


\bibitem{CW} M. Christ and M. Weinstein, Dispersion of small
amplitude solutions of the generalized Korteweg-de Vries equation.
J. Funct. Anal., 100(1991), 87--109.


\bibitem{CGT09} J. Colliander, M. Grillakis and N. Tzirakis,
\emph{Tensor products and correlation estimates with applications to
nonlinear Schr\"odinger equations.} Comm. Pure and Applied Math.,
\textbf{62} (2009), 920-968.


\bibitem{CR09} J. Colliander and P. Raphael, Rough blowup solutions to the $L^2$ critical NLS, Math. Anna., 345(2009), 307-366.


\bibitem{Dod} B. Dodson, Global well-posedness and scattering for the mass critical nonlinear Schr\"odinger equation with mass below the mass of the ground state. Adv. Math., 285(2015), 1589-1618.


\bibitem{FMR} G. Fibich, F. Merle and P. Raphael, Proof of a Spectral Property related to the singularity formation for the $L^2$
critical nonlinear Schr\"odinger equation. Physica D, 220(2006), 1-13.



\bibitem{GNN} B. Gidas, W.M. Ni, L. Nirenberg, Symmetry and related properties
via the maximum principle, Comm. Math. Phys. 68 (1979), 209-243.

\bibitem{GV79} J. Ginibre and G. Velo, On a class of nonlinear Schr\"odinger equations. I. The Cauchy problem, general case. J. Funct. Anal. 32(1979), 1-32.


\bibitem{GiV85b} J. Ginibre and G. Velo, Time decay of finite energy
solutions of the nonlinear Klein-Gordon and Schr\"{o}dinger
equations, Ann. Inst. H. Poincar\'{e} Phys. Th\'{e}or. 43 (1985),
399-442.



\bibitem{KeT98}M. Keel and T. Tao, Endpoint Strichartz estimates. Amer. J.
Math. 120:5 (1998), 955-980.



\bibitem{KTV2009} R. Killip, T. Tao, and M. Visan, The cubic nonlinear
Schr\"odinger equation in two dimensions with radial data. J. Eur.
Math. Soc., 11 (2009), 1203-1258.



\bibitem{KVZ2008} R. Killip, M. Visan, and X. Zhang, The mass-critical nonlinear
Schr\"odinger equation with radial data in dimensions three and
higher. Analysis and Partial Differential Equations, 1, no. 2
(2008) 229-266.


\bibitem{Kw} M.K. Kwong, Uniqueness of positive solutions of $\Delta u-u+u^p=0$ in $\R^n$,
Arch. Rational Mech. Anal. 105:3 (1989), 243-266.


\bibitem{Merle92}  F. Merle,
On uniqueness and continuation properties after blow-up time of self-similar solutions of nonlinear
Schr\"odinger  equation  with  critical  exponent  and  critical  mass
,  Comm.  Pure  Appl.  Math.  45  (1992),  no.  2,
203-254.

\bibitem{Merle93} F. Merle, Determination of blow-up solutions with minimal mass for nonlinear Schr\"odinger equations with critical power. Duke Math. J., 69(1993), 427-454.



\bibitem{MR05annmath} F. Merle and P. Raphael, The blow-up dynamic and upper bound
on the blow-up rate for critical nonlinear Schr\"odinger equation. Ann. Math., 161(2005), 157-222.


\bibitem{MR03GAFA}  F. Merle and P. Raphael, Sharp upper bound on the blow-up rate for the critical nonlinear Schr\"odinger equation.
Geom. Funct. Anal., 13(2003) 591-642.


\bibitem{MR04Invemath}  F. Merle and P. Raphael, On universality of blow-up profile for $L^2$ critical
nonlinear Schr\"odinger equation. Invent. math. 156(2004), 565-672.

\bibitem{MR06JAMS}  F. Merle and P. Raphael, On a sharp lower bound on the blow-up rate for the $L^2$ critical nonlinear Schr\"odinger equation.
J. AMS, 19(2005), 37-90.

\bibitem{MR05CMP}  F. Merle and P. Raphael, Profiles and Quantization of the Blow Up Mass
for Critical Nonlinear Schr\"odinger Equation. Commun. Math. Phys., 253(2005), 675-704.

\bibitem{MR05JHDE} F. Merle and P. Raphael, On one blow-up point solutions to the critical nonlinear Schr\"odinger equation. Jour. of Hype. Diff. Equa., 2(2005), 919-962.

\bibitem{MR08AJM} F. Merle and P. Raphael, Blow up of the critical norm for some radial $L^2$ super critical nonlinear
Schr\"odinger equations, Amer. J. Math. 130 (2008), 945-978.



\bibitem{MRS13AJM} F. Merle, P. Raphael and J. Szeftel, The instability of Bourgain-Wang solutions for the $L^2$ critical NLS, Amer. Jour. Math., 135(2013), 967-1017.



\bibitem{Ra05Mathann} P. Raphael, Stability of the log-log bound for blow-up solutions to the critical nonlinear Schr\"odinger equation. Math. Ann., 331(2005), 577-609.

\bibitem{Titch46} E.  C.  Titchmarsh,
Eigenfunction  expansions  associated  with  second-order  differential  equations,  Oxford,
Clarendon Press, 1946.


\bibitem{VZ07}  M. Visan and X. Zhang,  On the blowup for the $L^2$-critical focusing nonlinear Schr\"odinger equation in higher dimensions below the energy class. SIAM J. Math. Anal., 39(2007), 34-56.



\bibitem{wen82} M. I. Weinstein, Nonlinear Schr\"odinger equations and sharp interpolation estimates,
Comm. Math. Phys., 87(1983), 567-576.

\bibitem{Wen85} M.I. Weinstein, Modulational stability of ground states of nonlinear Schr\"odinger equations, SIAM J. Math. Anal., 16(1985), 567-576.

\bibitem{YRZ} K. Yang, S. Roudenko and Y. Zhao Blow-up dynamics and spectral property in the $L^2$-critical nonlinear Schr\"odinger equation in high dimensions,  arXiv:1712.07647.


\end{thebibliography}
\end{document}